\documentclass[11pt, reqno]{amsart}
\usepackage[T1]{fontenc}
\usepackage{yhmath,amsmath,amsfonts,amssymb,enumerate,amsthm,appendix,caption,lipsum,}
\usepackage{enumitem}
\usepackage[english]{babel}
\usepackage{cmap,mathtools}
\usepackage{graphics} 
\usepackage{epsfig} 
\usepackage{graphicx}\usepackage{epstopdf}

\usepackage{a4wide}
\usepackage[margin=0.95in]{geometry}

\usepackage{xcolor}
\usepackage{hyperref}
\usepackage[hyperpageref]{backref}
\hypersetup{
    colorlinks=true,
    linkcolor=myblue,
    filecolor=magenta,      
    urlcolor=mygreen,
    citecolor=mygreen,
}
\definecolor{mygreen}{rgb}{0.01,0.6,0.2}
\definecolor{myblue}{rgb}{0.01, 0.18, 1.0}
\newtheorem{theorem}{Theorem}
\newtheorem{proposition}[theorem]{Proposition}
\newtheorem{lemma}[theorem]{Lemma}

\theoremstyle{definition}
\newtheorem{definition}[theorem]{Definition}
\newtheorem{remark}[theorem]{Remark}
\newtheorem{example}[theorem]{Example}

\numberwithin{equation}{section}
\numberwithin{theorem}{section}
\numberwithin{equation}{section}
\numberwithin{theorem}{section}

\newcommand{\al} {\alpha}
\newcommand{\pa} {\partial}
\newcommand{\be} {\beta}
\newcommand{\de} {\delta}
\newcommand{\De} {\Delta}

\newcommand{\ga} {\gamma}
\newcommand{\Ga} {\Gamma}
\newcommand{\om} {\omega}
\newcommand{\Om} {\Omega}
\newcommand{\la} {\lambda}

\newcommand{\Gr} {\nabla}

\newcommand{\noi} {\noindent}
\newcommand{\ep} {\epsilon}

\newcommand{\ra} {\rightarrow}

\def\t{\tau}

\DeclarePairedDelimiter\abs{\lvert}{\rvert}%
\DeclarePairedDelimiter\norm{\lVert}{\rVert}%
\makeatletter
\let\oldnorm\norm
\def\norm{\@ifstar{\oldnorm}{\oldnorm*}}

\newcommand{\Dtwo}{\D^{2,2}(\RN)}
\newcommand{\intRn}{\displaystyle{\int_{\mathbb{R}^N}}}

\newcommand\restr[2]{{
  \left.\kern-\nulldelimiterspace 
  #1 
  \right|_{#2} 
  }}

\def\t{\tau}

\def\C{{\mathcal C}}
\def\D{{\mathcal D}}

\def\N{{\mathbb N}}
\def\F{{\mathcal F}}
\def\M{{\mathcal M}}
\def\S{\mathbb{S}}

\def\R{{\mathbb R}}
\def\RN{{\mathbb R}^N}

\def\({{\Big(}}
\def\){{\Big)}}

\def\ws2{{\F_{\frac{N}{2}}}}

\def\c1{{\C_c^1}}

\def\d{{\rm d}}

\def\ds{{\rm d}t}

\def\dx{{\rm d}x}
\def\dy{{\rm d}y}

\def\R{{\mathbb R}}

\def\C{{\mathcal C}}
\def\D{{\mathcal D}}
\def\M{{\mathcal M}}

\def\R{{\mathbb R}}
\def\N{{\mathbb N}}
\def\F{{\mathcal F}}

\def\ws2{{\F_{\frac{N}{2}}}}

\def\l2{{ L^{1,\;\infty}(\log L)^2}}
\usepackage[foot]{amsaddr}

\usepackage[pagewise]{lineno}

\title[On the fourth order semipositone problem in $\RN$]{On the fourth order semipositone problem in $\RN$}

\author[N. Biswas, U. Das, and A. Sarkar]{Nirjan Biswas$^1$, Ujjal Das$^2$, and Abhishek Sarkar$^{3,*}$}

\subjclass{35J35, 35J91, 35J08, 35B09, 35B65.}
\keywords{Semipositone problem, Biharmonic operator, Mountain-Pass solutions, Riesz potential, Positive solutions}

\thanks{$^*$Corresponding author.}

\begin{document}

\maketitle

\centerline{$^{1}$Tata Institute of Fundamental Research, Centre For Applicable Mathematics,}
\centerline{Post Bag No 6503, Sharada Nagar,}
\centerline{Bangalore 560065, India}
\centerline{$^{2}$Department of Mathematics, Technion - Israel Institute of Technology}
\centerline{Haifa 32000, Israel}
\centerline{$^{3}$Department of Mathematics, Indian Institute of Technology Jodhpur,}
\centerline{Rajasthan 342030, India}

\begin{abstract} 
For $N \geq 5$ and $a>0$, we consider the following semipositone problem
\begin{align*}
\De^2 u=  g(x)f_a(u)  \text {  in  }  \RN, \, \text{ and } \, u \in \D^{2,2}(\RN),\ \ \ \qquad \quad \mathrm{(SP)}
\end{align*}
where $g \in L^1_{loc}(\RN)$ is an indefinite weight function, $f_a:\R \to \R$ is a continuous function that satisfies $f_a(t)=-a$ for $t \in \R^-$, and  $\D^{2,2}(\RN)$ is the completion of $\C_c^{\infty}(\RN)$ with respect to $(\int_{\RN} (\De u )^2)^{1/2}$. For $f_a$ satisfying subcritical nonlinearity and a weaker Ambrosetti-Rabinowitz type growth condition, we find the existence of $a_1>0$ such that for each $a \in (0,a_1)$, (SP) admits a mountain pass solution. Further, we show that the mountain pass solution is positive if $a$ is near zero. For the positivity, we derive uniform regularity estimates of the solutions of (SP) for certain ranges in $(0,a_1)$, relying on the Riesz potential of the biharmonic operator.
\end{abstract} 

\section{Introduction }\label{intro}
 The present paper deals with the following fourth-order semilinear problem:  
\begin{align}\label{SP}\tag{SP}
\De^2 u=  g(x)f_a(u)  \text {  in  }  \RN, 
\end{align}
where $N \ge 5$, weight function $g$ is positive, $a>0$ and $f_a \in C(\R)$ is defined as
\begin{align*}
    f_a(t) = \left\{\begin{array}{ll} 
            f(t)-a , & \text {if }  t \ge 0; \\ 
            -a, & \text{if} \; t \le 0,  \\
             \end{array} \right. \text{ where } f \in C(\R^+) \text{ satisfying } f(0)=0.
\end{align*}
Further assumptions on the functions $g$ and $f$ will be stated
later. We look for the solutions of \eqref{SP} in the Beppo-Levi space $\Dtwo$-the completion of $\C_c^{\infty}(\RN)$ with respect to $\Vert \De u \Vert_2 = (\int_{\RN} (\De u )^2)^{1/2}$. A function $u \in \D^{2,2}(\RN)$ is said to be a weak solution of \eqref{SP}, if it satisfies the following identity: 
 \begin{align*}
     \intRn \De u \, \De v \, \dx = \intRn g(x)f_a(u)v \, \dx, \quad \forall \, v \in \D^{2,2}(\RN).
 \end{align*}
Notice that, the term $f_a(u)$ present in the right hand side of the equation \eqref{SP} is negative at $u \le 0$. The problem \eqref{SP} is known as \textit{semipositone} in the literature.

The semipositone problems appear in various contexts in mathematical physics, biology, engineering applications, etc. For example, in the logistic equation, mechanical systems, suspension bridges, chemical reactions, population model; see e.g., \cite{JS2010, OSS02} and the references therein. The semipositone problems were first observed in \cite[Brown and Shivaji]{BS83} while studying the bifurcation theory for the perturbed problem $-\Delta u =\lambda(u-u^3) -\epsilon, u > 0$ in $\Om$, with $\la, \ep >0$. In the literature, the second-order semipositone problems are well studied. Many authors considered the following problem on a bounded domain $\Om$:
\begin{align}\label{GSP}
-\Delta_p u = \lambda f(u)  \text {  in  }  \Om, \; u > 0 \text {  in  }  \Om, \, \text{ and } \, u=0  \text{  on  } \pa \Om,
\end{align}
where $\Delta_p$ is the $p$-Laplace operator defined as $\Delta_p(u) = \text{div}(\abs{\Gr u}^{p-2} \Gr u),$  $\la >0$, $f:\R^+ \ra \R$ is continuous, increasing and $f(0)<0$. The existence of solutions for \eqref{GSP} have been proved under various sufficient conditions on the nonlinearity $f$. For $p=2$, $N=1$ and $\Om= (0,1)$, in \cite{CS88} Castro and Shivaji introduced a sufficient condition, namely, $f \in C^2(\R^+)$, $f$ is superlinear, $f''>0$, and  $f(\al_1)=F(\al_2)=0$ for some $\al_1, \al_2 \in \R^+$, where $F$ is the primitive of $f$ defined as $F(t)=\int_0^t f(s) \, \ds$. For $p=2$ and $N \ge 2$, Castro et al. \cite{CHS95} assumed $f$ is smooth, concave, sublinear and unbounded near infinity.
For general $p$ and $N \ge 2$, \eqref{GSP} have been studied in \cite[Dancer and Zhang]{DZ05} where $f$ is $p$-sublinear, namely $f(t)=o(t^{p-1})$ as $t \ra \infty$, and bounded below by certain function, in \cite[Chhetri et al.]{CDS15} where $f$ is $p$-superlinear, namely $f(t)=O(t^{\ga}); \ga \in (p-1, \frac{N(p-1)}{N-p})$ as $t \ra \infty$. Several authors considered \eqref{GSP} for $f \equiv f(x,u) \in C_{loc}^{\al}(\R^{N+1})$ with $f(x,0)<0$ and imposed sufficient conditions on $f(x,\cdot)$ for the existence of nontrivial solutions. For instance, see \cite[Allegretto et al.]{ANZ92} where $f(x,\cdot)$ is superlinear and subcritical.
 For more existing results related to second-order semipositone problems on a bounded domain, we refer \cite{CFL16, CQT17} and the references therein. In order to obtain the existence of solutions, the standard tools used are topological degree, sub-super solution methods, variational methods, and fixed point theory. The authors mainly utilized the regularity and some versions of the strong maximum principle to get the positive solutions. In general, proving the existence of positive solutions for semipositone problems is challenging.

In \cite{AHS20}, Alves et al. studied the following class of semipositone problems:
\begin{align}\label{SP2}
    -\De u = h(x)(f(u)-a) \text{ in } \R^N,  \text{ and } \, u \in \D^{1,2}(\RN),
\end{align}
where $\D^{1,2}(\RN)$ is the completion of $\C_c^{\infty}(\RN)$ with respect to $\Vert \Gr u \Vert_2$, $f \in C(\R^+)$ is locally Lipschitz that satisfies $f(0)=0$, superlinear growth conditions and \eqref{RA} type nonlinearties, $h \in L^1(\RN) \cap L^{\infty}(\RN)$ is positive, radial and satisfies the following property
\begin{align}\label{Alves1}
 |x|^{N-2} \intRn \frac{h(|y|)}{|x-y|^{N-2}} \, \dy \le C_h \quad \text{for } x \in \RN \setminus \{ 0 \}, \text{ where } C_h>0.
\end{align}
The authors provided a specific range of $a$ for which \eqref{SP2} admits a positive solution. For the existence of solutions, variational methods were used. To get the positivity of solutions, the authors first obtained the regularity and uniform boundedness of the solutions of \eqref{SP2} for $a$ close to zero, using the regularity estimate by Brezis and Kato in \cite{BK1979}, and then used the Riesz potential for \eqref{SP2}, together with the strong maximum principle and \eqref{Alves1}.

Next, we shift our attention to the higher-order semipositone type problems. In the last two decades, the study of higher-order differential elliptic operators got attention from both the mathematical point of view and its applications in science and engineering. For example, in the study of elastic (hinged) bending beam under a possible nonlinear loading \cite{Truesdell91}, in the steady-state prototype equation of phase transition in condensed matter systems \cite{Uzunov1993}, thin-film equation \cite{DPGS2001}, etc., to name a few. In the literature, a few authors studied fourth-order semipositone type problems for the unidimensional case. For example, in \cite{Ma03} where the existence of multiple positive solutions was shown while dealing with beam equations that arise in structural engineering.

Motivated by the fourth-order boundary value problems with real-world applications, we consider \eqref{SP} for both its mathematical essences and also for the aspects of its application. To the best of our knowledge, the higher-order semipositone problem in the higher dimensions has not been considered yet. The main difficulties in studying the positivity of the solutions for such problems with the Dirichlet boundary condition are the lack of maximum principle and the regularity results for weak solutions. On the other hand, fourth-order semipositone problems with Navier boundary condition can be tackled by the study of second order semipositone systems (see \cite{HS2004}).

In this article, we consider the following hypothesis on $f$ and $g$:
\begin{enumerate}[label={($\bf f{\arabic*}$)}]
\setcounter{enumi}{0}
    \item \label{f1} $\displaystyle\lim_{t \rightarrow 0}\frac{f(t)}{t}=0$ and 
    $\displaystyle\lim_{t \rightarrow \infty} \frac{f(t)}{t^{\ga-1}} \le C_f $ for some $\gamma\in \left(2, 2^{**}\right)$ and $C_f>0$,  \vspace{0.2 cm}
    \item \label{f2} $\displaystyle \lim_{t \rightarrow \infty}\frac{f(t)}{t}=\infty$,  \vspace{0.1 cm}
    \item \label{f3}  there exists  $R>0$ such that  $\displaystyle \frac{f(t)}{t}$ is increasing for $t>R$, \vspace{0.2 cm}
    \item \label{f4} $f$ is locally Lipschitz. 
\end{enumerate}
and
\begin{enumerate}[label={($\bf g{\arabic*}$)}]
\setcounter{enumi}{0}
     \item \label{g1} $g \in L^1(\RN) \cap L^{\infty}(\RN)$,  \vspace{0.2 cm}
     \item \label{g2} for $\de = 1$ and $2^{**}$, $|x|^{(N-4)\de} \intRn \frac{g(y)^{\de}}{|x-y|^{(N-4)\de}} \, \dy \le C_g$, where $x \in B_1^c$ and  $C_g>0$,
\end{enumerate}
where $2^{**}=\frac{2N}{N-4}$ and $B_1^c$ is the exterior of the open unit ball in $\RN$ with centre at the origin. In Example \ref{f and g}, we provide $f,g$ satisfying \ref{f1}-\ref{f4} and \ref{g1}-\ref{g2}. Under the above assumptions we prove the following results: 
\begin{theorem}\label{result}
Let $f$ satisfies {\rm \ref{f1}-\ref{f3}}, and let $g$ be positive and satisfies {\rm \ref{g1}}. Then the following hold:

\begin{enumerate}
    \item[(i)] {\rm \textbf{(Existence)}} There exists $a_1>0$ such that for each $a \in (0, a_1)$, \eqref{SP} admits a mountain pass solution $u_a \in C(\RN) \cap L_{loc}^{\infty}(\RN)$.
    \item[(ii)] {\rm \textbf{(Regualrity and uniform boundedness)}}  There exist $a_2 \in (0, a_1)$ and $C>0$, such that
\begin{align*}
    \norm{\De u_a}_{2} \le C, \; \forall \, a \in (0, a_2).
\end{align*}
In addition, if $g$ satisfies {\rm \ref{g2}}, then for every $a \in (0, a_2)$, $u_a \in L^{\infty}(\RN)$. Moreover, there exists $C>0$ such that $$\norm{u_a}_{L^{\infty}(B_1^c)} \le C, \; \forall \, a \in (0, a_2).$$
    \item[(iii)] {\rm \textbf{(Positivity)}} There exists $a_3 \in (0, a_2)$ such that
for every $a \in (0, a_3)$, $u_a \ge 0$ on $\RN$. Furthermore, if $f$ satisfies {\rm \ref{f4}}, then $u_a>0$ a.e. on $\RN$.
\end{enumerate}
\end{theorem}

For the existence of solution of \eqref{SP}, we are not considering the Ambrosetti-Rabinowitz (RA) type nonlinearities for $f$. In \cite{AR73} Ambrosetti and Rabinowitz studied the existence of nontrivial solutions to the problem:
\begin{equation}\label{semilinear}
    -\De u=\la f(x,u) \; \text{ in } \Omega, \;  u=0 \; \text{ on } \partial\Omega,
\end{equation}
assuming the following condition namely: 
\begin{equation}\tag{RA}\label{RA}
    \begin{split}
        \text{there exist } \theta>2 \text{ and } t_0>0 \text{ such that } \\
    0<\theta F(x,t)\le tf(x,t), \; \forall \,  |t|\ge t_0, x \in \Om,
    \end{split}
\end{equation}
where $f(x,t)$ is a continuous function on $\overline{\Om}\times \R$ satisfying $f(x,0)=0$ and certain growth conditions, and $F(x,t)=\int_0^t f(x,s) \, \ds$. The condition \eqref{RA} appears in most of the semilinear elliptic PDEs to show the existence of nontrivial solutions. The advantages of the (RA) type nonlinearities are two-fold. Firstly, it makes sure that the Euler-Lagrangian functional associated with the boundary value problem \eqref{semilinear} has a mountain pass geometry; secondly, it ensures that the Palais-Smale sequence of the Euler-Lagrangian functional is bounded. But this condition is restrictive to some extent in the sense that many nonlinearities such as $f(t)=2t\ln (1+|t|)$ do not satisfy \eqref{RA}. Several authors later replaced \eqref{RA} with the following weaker condition: $$\lim_{|t|\ra \infty}\frac{F(x,t)}{t^2}=+\infty.$$ Among them, we would like to refer the works of \cite[Miyagaki and Souto]{MS08}, \cite[Lan and Tang]{LT2014}, and \cite[Iturriaga et al.]{ILU2010}, where the authors also implemented some monotonicity arguments on $\frac{f(x,t)}{t}$. In this article, the similar assumptions that we imposed are \ref{f2} and \ref{f3}.

Now let us briefly describe the technique used for the existence of positive solutions of \eqref{SP}. One of the key ingredients for proving the positivity is a weaker Brezis-Kato type regularity for the biharmonic operator, which is recently obtained by Medereski and Siemianowski \cite{Medereski}. Using this, we can show that every mountain pass solution of \eqref{SP} is locally bounded and continuous (see Remark \ref{bounded1.1}). The novelty of our approach is to prove the uniform boundedness of the mountain pass solutions $(u_a)$ of \eqref{SP} in $\Dtwo$ for $a$ near $0$, and then employ this along with the Riesz potential of \eqref{SP}, and the condition \ref{g2} to establish that $u_a \in L^{\infty}(\RN)$ for each $a$ close to $0$, and $(u_a)$ is uniformly bounded in $L^{\infty}(B^c_1)$. Applying these, we have shown that $(u_a)$ converges to a function in $L^{\infty}(B^c_1)$, as $a$ goes to $0$, and the limiting function is a positive solution to a positone problem. These help us to get $(u_a)$ is positive.

The rest of this article is organized as follows. In Section \ref{sec1}, we recall the higher-order Beppo-Levi space and discuss various embedding results. Then we briefly describe some properties of the Riesz potential for biharmonic operator. Section \ref{sec2} is devoted to setting up a functional framework for the existence of solutions of \eqref{SP}. Section \ref{sec3} contains the proof of the existence and various qualitative properties of the solutions for \eqref{SP}. In that section, we prove Theorem \ref{result}.

\section{Preliminaries}\label{sec1}
In this section, we define the Beppo-Levi space, discuss various embedding results of $\Dtwo$, and briefly describe some properties of the fundamental solution of the biharmonic operator that we frequently use in the subsequent sections. First, we list some of the notations used in this article. 
\begin{itemize}
    \item $B_R(x)$ is the open ball in $\RN$ with radius $R$ and centre at $x$. $B_R^c(x)$ represents the complement of $B_R(x)$ in $\RN$, i.e., $B_R^c(x) = \RN \setminus B_R(x)$. When $x=0$, we denote $B_R(0)$ as $B_R$ and $B_R^c(0)$ as $B_R^c$.
    \item $K$ is denoted as compact set in $\RN$.
    \item We denote  Riesz constants: $\mathcal{R}_4=\frac{1}{2(N-2)(N-4)\om_N}$ and $\mathcal{R}_2=\frac{1}{(N-2)\om_N}$, where $\omega_N$ is the $(N-1)$-dimensional measure of the Euclidean sphere $\mathbb{S}^N$, i.e., $\omega_N:=\displaystyle\frac{N\pi^{N/2}}{\Gamma\left( 1+\frac{N}{2}\right)}.$
    \item $C$ represents positive constant which may differ from one line to another. 
\end{itemize}

\subsection{Embeddings of $\Dtwo$}
In this subsection, we discuss some classical embeddings of the Beppo-Levi space $\Dtwo$. Recall that, for $m \in \mathbb{N}, 1<p<\infty$ and $N > mp$, the Beppo-Levi space $\mathcal{D}^{m,p}(\RN)$ is defined as:
\begin{align*}
    \mathcal{D}^{m,p}(\RN) := \text{completion of } C_c^{\infty}(\RN) \text{ with respect to } \norm{u}:=\left(\int_{\RN} |\Gr^m u|^p \right)^{\frac{1}{p}},
\end{align*}
where \[ \Gr^m u :=\begin{cases} \De^{m/2} u, \ m \text{ even }\\ \nabla \De^{(m-1)/2}u, \ m \text{ odd. }\end{cases}\]
We note that, when $N \leq mp$, $\mathcal{D}^{m,p}(\mathbb{R}^N)$ may not be a function space (see \cite[Remark 2.2]{Tintarev2007}). For more details on the Beppo-Levi spaces, we refer to the book \cite{Tintarev2007}. In the following proposition, we show that for $N \ge 5$, $\Dtwo$ is continuously embedded in $\D^{1,2^*}(\RN)$ (where $2^*=\frac{2N}{N-2}$) and $W^{2,2}_{loc}(\RN)$. 
\begin{proposition}\label{local}
Let $N \ge 5$. Then 
\begin{enumerate}
    \item[(i)] $\Dtwo \hookrightarrow \D^{1,2^*}(\RN)$,
    \item[(ii)] $\Dtwo \hookrightarrow W^{2,2}_{loc}(\RN)$,
    \item[(iii)] $\Dtwo$ is compactly embedded into $L_{loc}^{\ga}(\RN)$ where $\ga \in [1,2^{**})$.
\end{enumerate}
\end{proposition}
\begin{proof}
$(i)$ Let $u \in \C_c^{\infty}(\RN)$. Then
\begin{equation*}
    \begin{split}
        \intRn |\Gr u|^{2^*}  =  \intRn \left( \sum_{i=1}^N \left(\frac{\pa u}{\pa x_i}\right)^2 \right)^{\frac{2^*}{2}}  \le 2^{\frac{2}{N-2}} \intRn \sum_{i=1}^N \left| \frac{\pa u}{\pa x_i} \right|^{2^*}.
    \end{split}
\end{equation*}
Hence using the embedding $\D^{1,2}(\RN) \hookrightarrow L^{2^*}(\RN)$ (\cite[Theorem 2.4]{Tintarev2007}), we obtain
\begin{align*}
    \left( \intRn |\Gr u|^{2^*} \right)^{\frac{2}{2^*}} & \le 2^{\frac{2}{N}} \left( \intRn \sum_{i=1}^N \left| \frac{\pa u}{\pa x_i} \right|^{2^*}  \right)^{\frac{2}{2^*}}  \\
    & \le C \left(\sum_{i=1}^N \left(\intRn \left| \Gr \left( \frac{\pa u}{\pa x_i} \right) \right|^2  \right)^{\frac{2^*}{2}} \right)^{\frac{2}{2^*}}  \\
    & \le C \left( \left( \sum_{i=1}^N \left(\intRn \left| \Gr \left( \frac{\pa u}{\pa x_i} \right) \right|^2 \right) \right)^{\frac{2^*}{2}} \right)^{\frac{2}{2^*}} \\
    & = C \intRn \sum_{i=1}^N \left| \Gr \left( \frac{\pa u}{\pa x_i} \right) \right|^2  \\
    & = C \left( \intRn (\De u)^2 + \intRn \sum_{i=1}^N \sum_{j=1,j\neq i}^N  \left(\frac{\partial^2 u}{\partial x_i \partial x_j}\right)^2 \right),
\end{align*}
where $C=C(N)$. Now by using the Green's theorem,
\begin{align*}
   \intRn  \left( \frac{\partial^2 u}{\partial x_i \partial x_j} \right)^2 = \intRn \left( \frac{\pa^2 u}{\pa x_i^2}\right) \left(\frac{\pa^2 u}{\pa x_j^2} \right) \le \frac{1}{2} \intRn \left(\frac{\pa^2 u}{\pa x_i^2}\right)^2 + \left(\frac{\pa^2 u}{\pa x_j^2} \right)^2.
\end{align*}
The above inequality yields 
\begin{align*}
    \intRn \sum_{i=1}^N \sum_{j=1,j\neq i}^N  \bigg(\frac{\partial^2 u}{\partial x_i \partial x_j}\bigg)^2 \le C \intRn (\De u)^2,
\end{align*}
where $C=C(N)$. Therefore, we get 
\begin{align*}
     \left( \intRn |\Gr u|^{2^*} \right)^{\frac{2}{2^*}} \le C \intRn (\De u)^2, \; \forall \, u \in C_c^{\infty}(\RN).
\end{align*}
By the density, we conclude that $\Dtwo \hookrightarrow \D^{1,2^*}(\RN)$. 

$(ii)$ From the embeddings $\Dtwo \hookrightarrow L^{2^{**}}(\RN)$ (see \cite[(2.6) page-58]{Tintarev2007}) and  $L^{2^{**}}(\RN) \hookrightarrow L_{loc}^2(\RN)$, it follows that $\Dtwo \hookrightarrow L_{loc}^2(\RN)$. For any $K \subset \R^N$, using the H\"{o}lder's inequality and $\Dtwo \hookrightarrow \D^{1,2^*}(\RN)$, we get 
\begin{align*}
    \int_K |\Gr u|^2 \le C\left( \int_K |\Gr u|^{2^*} \right)^{\frac{2}{2^*}} \le C \intRn (\De u)^2, \; \forall \, u \in \Dtwo,
\end{align*}
where $C=C(N,K)$. Further, using the Green's theorem, there exists $C=C(N)$ such that
\begin{align*}
    \int_{\RN} |\Gr^2 u|^2 \le C \int_{\RN}(\De u)^2, \; \forall \, u \in \Dtwo.
\end{align*}
Therefore, for any $K \subset \RN$,
\begin{align*}
    \int_K u^2 + |\Gr u|^2 + |\Gr^2 u|^2 \le C\intRn (\De u)^2, \; \forall \, u \in \Dtwo.
\end{align*}
Thus $\Dtwo \hookrightarrow W^{2,2}_{loc}(\RN)$. 

$(iii)$ Using the Rellich-Kondrachov compactness theorem, $W^{2,2}_{loc}(\RN)$ is compactly embedded into $L^{\ga}_{loc}(\RN)$ (see \cite[Corollary 6.1]{Nevcas2012}). Therefore, using $\Dtwo \hookrightarrow W^{2,2}_{loc}(\RN)$ the required compactness holds.
\end{proof}

\begin{remark}\label{classical embedding}
For each $\tilde{\ga} \in [2,2^{**}]$, let $g \in L^{\tilde{p}}(\RN)$ where $\frac{1}{\tilde{p}} + \frac{\tilde{\ga}}{2^{**}} = 1$. Then the Sobolev embedding $\Dtwo \hookrightarrow L^{2^{**}}(\RN)$ and the H\"{o}lder's inequality ensure that
\begin{align}\label{cembed2}
    \intRn \abs{g} |u|^{\tilde{\ga}} \le \norm{g}_{\tilde{p}} \norm{|u|^{\tilde{\ga}}}_{\frac{2^{**}}{\tilde{\ga}}} \le C \norm{g}_{\tilde{p}}\left( \intRn (\De u)^2  \right)^{\frac{\tilde{\ga}}{2}}, \; \forall \, u \in \D^{2,2}(\RN),
\end{align}
where $C=C(N)$ is the embedding constant. Thus the embedding $\Dtwo \hookrightarrow L^{\tilde{\ga}}(\RN,|g|)$ holds.
\end{remark}

In the next proposition, we see that the above embedding is compact for $\tilde{\ga} \neq 2^{**}$. 

\begin{proposition}\label{compact embeddings}
For each $\tilde{\ga} \in [2,2^{**})$, let $g \in L^{\tilde{p}}(\RN)$ where $\frac{1}{\tilde{p}} + \frac{\tilde{\ga}}{2^{**}} = 1$. Then the embedding $\Dtwo \hookrightarrow L^{\tilde{\ga}}(\RN,|g|)$ is compact.
\end{proposition}

\begin{proof}
Let $u_j \rightharpoonup u$ in $\Dtwo.$ Set $L = \sup \{ \| \De (u_j - u) \|_2^{\tilde{\ga}} \}.$ Let $\ep > 0$ be given. By the density of $\C_c(\RN)$  in $L^{\tilde{p}}(\RN)$, there exists $g_{\ep} \in \C_c(\RN)$ such that 
$\| g - g_{\ep} \|_{\tilde{p}} < \frac{\ep}{L}.$
We estimate
\begin{align}\label{cembed2.1}
\intRn |g| |u_j - u|^{\tilde{\ga}} \leq \intRn |g - g_{\ep}| |u_j - u|^{\tilde{\ga}} + \intRn |g_{\ep}| |u_j - u|^{\tilde{\ga}}.
\end{align}
Using \eqref{cembed2} we have  
\begin{align}\label{cembed2.2}
\intRn |g - g_{\ep}| |u_j - u|^{\tilde{\ga}} \leq  C \| g - g_{\ep} \|_{\tilde{p}} \| \De (u_j - u) \|_2^{\tilde{\ga}} < C \ep.
\end{align} 
By the compactness of the embedding of $\Dtwo$ into $L^{\tilde{\ga}}_{loc}(\RN)$ (Proposition \ref{local}-$(iii)$), there exists $j_1 \in \N$ such that
\begin{align*}
   \intRn |g_{\ep}| |u_j - u|^{\tilde{\ga}}  = \int_{K} |g_{\ep}| |u_j - u|^{\tilde{\ga}} < \ep , \; \forall \, j > j_1,
\end{align*}
where $K$ is the support of $g_{\ep}$. Therefore, from \eqref{cembed2.1} and \eqref{cembed2.2} we obtain $$\intRn |g| |u_j - u|^{\tilde{\ga}} < (C+1)\ep, \; \forall \, j \ge j_1.$$ Thus $u_j \ra u$ in $L^{\tilde{\ga}}(\RN,|g|)$, as required. 
\end{proof}

\begin{remark}
If $g$ satisfies \ref{g1}, then $g \in L^q(\RN)$ for any $q \in [1, \infty]$. Thus, the results of Remark \ref{classical embedding} and Proposition \ref{compact embeddings} hold for such $g$.
\end{remark}

\subsection{Riesz potential}
For $N \ge 5$, $f(x) = \mathcal{R}_4 \abs{x}^{4-N}, \; x \in \RN$ is the fundamental solution of $\De^2$ in $\RN$, i.e., $\De^2 f = \de_0$ a.e. in $\RN$, where $\de_0$ is the Dirac delta function at the origin. This subsection considers some properties of $\abs{x}^{4-N}$. First, we recall the weak-$L^p$ space. 

\begin{definition}
Let $M(\RN)$ be the set of all extended real valued Lebesgue measurable functions that are finite a.e. in $\RN$. Let $f \in M(\RN)$. For $p \in (1, \infty)$, we consider the following quantity:
\begin{align*}
    \abs{f}_{p,w} = \sup_{s>0} s \left| \left\{ x \in \RN : \abs{f(x)} > s \right\} \right|^{\frac{1}{p}},
\end{align*}
where $|\cdot|$ denotes the Lebesgue measure. The weak-$L^p$ space $L^p_{w}(\RN)$ is defined as  
\begin{align*}
    L^p_{w}(\RN) = \left\{f \in  M(\RN) : \abs{f}_{p,w} < \infty \right\}.
\end{align*}
\end{definition}

Let $1<p,q < \infty$. For $f \in L^p_{w}(\RN)$ and $h \in L^q(\RN)$, the convolution $f*h$ is defined as 
\begin{align*}
    (f*h)(x) = \intRn f(x-y) h(y)\, \dy, \; \text{ where } x\in \RN. 
\end{align*}
In the following proposition, we state the weak Young's inequality involving convolution. For more details and proofs of this inequality, we refer to the book \cite{Lieb2001}. 

\begin{proposition}\label{convolution weak}
Let $1<p,q ,r < \infty$ be such that $1+\frac{1}{r}=\frac{1}{p} + \frac{1}{q}$. Let $f \in L^p_{w}(\RN)$ and $h \in L^q(\RN)$. Then $f*h \in L^r(\RN)$, and
\begin{align*}
    \norm{f*h}_{r} \le C(N,p,q) \norm{f}_{p,w} \norm{h}_q.
\end{align*}
\end{proposition}

\begin{remark}\label{Hardy potential}
 Let $N \ge 5$. First we see that $f(x)=|x|^{4-N} \in L^{\frac{N}{N-4}}_{w}(\RN)$. For any $s>0$, we calculate
\begin{align*}
    \left| \left\{x \in \RN: f(x) > s \right\} \right| =  \left| \left\{x \in \RN: \abs{x} < s^{\frac{-1}{N-4}} \right\} \right| = \frac{\om_N}{N} s^{\frac{-N}{N-4}}.
\end{align*}
Hence $\sup_{s>0} s^{\frac{N}{N-4}} \left| \left\{x \in \RN: f(x) > s \right\} \right| = \frac{\om_N}{N}.$ Therefore, $|x|^{4-N} \in L^{\frac{N}{N-4}}_{w}(\RN)$. Now by Proposition \ref{convolution weak}, for any $h \in L^q(\RN)$ with $1<q<\frac{N}{4}$, we have  $f*h \in L^r(\RN)$ where $r= \frac{Nq}{N-4q}$.
\end{remark}

Now we state a Riesz potential theorem for the biharmonic operator. For more details on this topic, we refer to \cite{DJ2015, SGG10}.

\begin{proposition}\label{fundamental}
Let $N \ge 5$. Let $h \in L^q(\RN)$ with $1<q<\frac{N}{4}$. We assume that $u\in \Dtwo $ is a weak solution of the following problem:
\begin{align}\label{RR1}
    \De^2 u = h \text{ in } \RN. 
\end{align}
Then 
\begin{align*}
     u(x) = \mathcal{R}_4 \intRn \frac{h(y)}{|x-y|^{N-4}} \, dy \; \text{ and } \;  -\De u(x) &= \mathcal{R}_2 \intRn \frac{h(y)}{|x-y|^{N-2}} \, dy \; \text{ a.e. in } \RN.
\end{align*}
\end{proposition}

\begin{proof}
Proof follows using \cite[Proposition 5.1]{DJ2015}, \cite[Remark 5.2]{DJ2015}, and Proposition \cite[Proposition 5.4]{DJ2015}. 
\end{proof}

\begin{remark}
From Proposition \ref{fundamental} and Remark \ref{Hardy potential} we see that any weak solution of \eqref{RR1} lies in $L^{\frac{Nq}{N-4q}}(\RN);1<q<\frac{N}{4}$. 
\end{remark}

\section{Functional framework}\label{sec2}
In this section, we develop a suitable functional settings for showing the existence of positive solutions of \eqref{SP}. We study the Cerami condition and the mountain pass geometry of the energy functional associated with \eqref{SP}.

\begin{remark}\label{bound}
In this remark, using the growth condition of $f$, we identify certain bounds for $f$ and $F$, required in the subsequent sections.  

\noi $(i)$ Let $\ep>0$ and $\ga \in (2,2^{**}]$. Using \ref{f1}, there exists $t_1>0$ such that 
\begin{align*}
    f(t) < \ep t, \text{ for } 0<t<t_1, \; \text{ and } \; f(t) \le C t^{\ga-1}, \text{ for } t \ge t_1. 
\end{align*}
where $C=C(C_f,t_1(\ep))>0$. Therefore, $f(t) \le \ep t + Ct^{\ga-1}$, for all $t \in \R^+$, and $f_a(t) \le \ep |t| + C|t|^{\ga -1} - a$, for all $t \in \R$. Also from \ref{f1}, we see that $f(t) \le C_f t^{\ga-1}$ for $t > t_2$ and hence $f(t) \le C(1+t^{\gamma -1})$, for $t \in \R^+$ where $C=C(C_f,t_2)$. Moreover, if $a \in (0, \varrho)$, then 
\begin{align*}
    |f_a(t)| \le \left\{ \begin{array}{ll}
    C(1+t^{\gamma -1})+\varrho, & \, \text{  for  }  t \ge 0; \vspace{0.1 cm} \\
     \varrho, &  \, \text{  for  }  t \le 0.
\end{array}\right.
\end{align*}
Hence $|f_a(t)| \le C(1+|t|^{\gamma -1})$, for $t \in \R$  where $C=C(C_f,t_2, \varrho)$.

\noi $(ii)$ Let $M>0$. Due to \ref{f2}, there exists $C_M>0$ such that $f(t) > Mt-C_M$, for all $t \in \R^+$. Also, using \ref{f2}, we get  $\lim_{t \rightarrow \infty}\frac{F(t)}{t^2}=\infty$. Hence $F(t)>Mt^2-C_M$, for all  $t \in \R^+$ and for some $C_M>0$.
\end{remark}

\begin{remark}\label{Riesz0.1}
Let $u \in \Dtwo$ and $g$ satisfies \ref{g1}. In this remark, we see that $gf_a(u) \in L^q(\RN)$ for $1<q<\frac{2N}{N+4}$. Using Remark \ref{bound}-$(i)$ we calculate
\begin{align*}
    \intRn g^q|f_a(u)|^q \le C2^{q-1} \intRn g^q \left(1+ |u|^{(\ga-1)q}\right) \le C \norm{g}_q + C \intRn g^q |u|^{(\ga-1)q}. 
\end{align*}
Note that for $q \in (1, \frac{2N}{N+4})$, $(\ga-1)q < 2^{**}$. Set $r = \frac{2^{**}}{(\ga-1)q}$. Then using the H\"{o}lder's inequality with the conjugate pair $(r,r')$ to get
\begin{align*}
    \intRn g^q |u|^{(\ga-1)q} \le \left( \intRn |u|^{2^{**}} \right)^{\frac{1}{r}} \left( \intRn g^{qr'} \right)^{\frac{1}{r'}} 
    \le C \norm{g}_{qr'}^q \left( \intRn (\De u)^2 \right)^{\frac{N}{(N-4)r}},
\end{align*}
where $C=C(N)$. Thus $gf_a(u) \in L^q(\RN)$ for $1<q<\frac{2N}{N+4}$.
\end{remark}

In the following proposition, we establish an inequality involving $f_a$ and it's primitive. 
\begin{proposition}\label{relation}
Let $f$ satisfies {\rm\ref{f1}} and {\rm \ref{f3}}. Then there exists $C_R>0$ such that 
\begin{align*}
    F_a(s)-F_a(t) \le \frac{s^2-t^2}{2s}f_a(s) + C_R, \quad \forall \, 0<t \le s \text{ or, } s \le t<0.
\end{align*}
\end{proposition}
\begin{proof}
Let $0<t \le s$. First, we consider $R \le t \le s$. Then using the fundamental theorem of integration, we write
\begin{align*}
    F_a(s)- F_a(t)= \int_{t}^s f_a(\tau) \d \tau = \int_{t^2}^{s^2} \frac{f_a(\tau^{\frac{1}{2}})}{2 \tau^{\frac{1}{2}}} \d \tau.  
\end{align*}
Using \ref{f3}, we see that for $\tau >R$, $\frac{f_a(\tau)}{\tau}= \frac{f(\tau)}{\tau}-\frac{a}{\tau}$ is increasing.  Hence, from the above identity we get $$F_a(s)-F_a(t) \le  \frac{s^2-t^2}{2s}f_a(s).$$ Suppose $t \le s \le R$. We use the fundamental theorem of integration and Remark \ref{bound}-$(i)$, to obtain
\begin{align*}
    F_a(s)- F_a(t) \le \int_{t}^s \left( \ep \t+ C \t^{\ga -1} - a\right) \, \d\t \le \frac{\ep}{2} s^2 + \frac{C}{\ga} s^{\ga} \le C_R
\end{align*}
where $C_R= \frac{\ep}{2} R^2 + \frac{C}{\ga} R^{\ga}$. Next, we assume $0< t \le R \le s$. Using the above estimates 
\begin{align*}
    F_a(s)-F_a(t) & =F_a(s)-F_a(R)+F_a(R)-F_a(t) \\
    &\le \frac{s^2-R^2}{2s}f_a(s)+C_R \\
    & \le \frac{s^2-t^2}{2s}f_a(s)+C_R.
\end{align*}
Therefore, 
$$F_a(s)-F_a(t) \le \frac{s^2-t^2}{2s}f_a(s) + C_R, \; \forall \, 0<t \le s.$$ 
If $s \le t<0$, then 
$$F_a(s)-F_a(t)=a(t-s) = \frac{a(t^2-s^2)}{s+t} \le \frac{a(t^2-s^2)}{2s} = f_a(s) \frac{s^2-t^2}{2s}.$$ 
This completes our proof.
\end{proof}

For $g$ satisfying \ref{g1}, we consider the following functionals on $\Dtwo$:
\begin{align*}
    G(u)= \intRn  g u^2, \; N_a(u) = \intRn g F_a(u), \text{ and } I_a(u) = \frac{1}{2} \intRn (\De u)^2 - N_a(u). 
\end{align*}
From Remark \ref{classical embedding} and using $|F_a(u)| \le C(u^2 + |u|^{\ga}) + a|u|$ (by Remark \ref{bound}-$(i)$) we can verify that $G,N_a$, and $I_a$ are well-defined. Moreover, $G, N_a, I_a \in C^1(\D^{2,2}(\RN), \R)$ and for $u, v \in \D^{2,2}(\RN)$,
\begin{align*}
    \left< G'(u), v \right> &= 2\intRn gu v, \; \left< N_a'(u), v \right> = \intRn gf_a(u)v, \text{ and }  \\ 
    \left< I_a'(u), v \right> &= \intRn \De u \, \De v - \left< N_a'(u), v \right>,
    \end{align*}
where $\left< \, , \right>$ denotes the duality action.
\begin{proposition}\label{compact map}
Let  $g$ be positive and satisfies {\rm\ref{g1}}, and $f$ satisfies {\rm \ref{f1}}. Then the following holds:
\begin{enumerate}
    \item[(i)] $G$ is compact on $\D^{2,2}(\RN).$ 
    \item[(ii)] $N_a$ is compact on $\D^{2,2}(\RN).$
    \item[(iii)] Let $u_j \rightharpoonup u$ in $\D^{2,2}(\RN)$ and $a_j \ra a$ in $\R^+$. Then $N_{a_j}(u_j) \ra N_a(u)$.
\end{enumerate}
\end{proposition}
\begin{proof}
$(i)$ The compactness of $G$ follows from Proposition \ref{compact embeddings}. 

$(ii)$ Let $u_j \rightharpoonup u$ in $\D^{2,2}(\RN)$. It is required to show $N_a(u_j) \ra N_a(u)$.  Let $\ep > 0$ be given and $\tilde{p}$ be the conjugate exponent of $\frac{2^{**}}{\tilde{\ga}}$ where $\tilde{\ga} \in [2,2^{**}]$. By the density of $\C_c(\RN)$ into $L^{\tilde{p}}(\RN)$, split $g= g_{\ep} + (g-g_{\ep})$ where $g_{\ep} \in \C_c(\RN)$ such that $|g_{\ep}| \le g$ and $\norm{g-g_{\ep}}_{\tilde{p}}<\ep$. We write
\begin{align}\label{compact1}
    \abs{N_a(u_j)-N_a(u)} \le \intRn \abs{g-g_{\ep}}|F_a(u_j)-F_a(u)| + \intRn |g_{\ep}||F_a(u_j)-F_a(u)|.
\end{align}
First we estimate the second integral of \eqref{compact1}. Let $K$ be the support of $g_{\ep}$. Since $\D^{2,2}(\RN)$ is compactly embedded into $L^{\ga}_{loc}(\RN)$ for $\ga \in [1,2^{**})$ (Proposition \ref{local}-$(ii)$), $u_j \rightarrow u$ in $L^{\ga}(K)$. In particular, up to a subsequence, $u_j(x) \rightarrow u(x)$ a.e. in $K$. From the continuity of $F_a$, $F_a(u_j(x)) \ra F_a(u(x))$ a.e. in $K$. Further, from Remark \ref{bound}-$(i)$, $|F_a(u_j)| \le C(u_j^2 + |u_j|^{\ga}) + a|u_j|$. Therefore, using the generalized dominated convergence theorem, $F_a(u_j) \ra F_a(u)$ in $L^1(K)$. Hence  
\begin{align}\label{compact2}
    \intRn |g_{\ep}||F_a(u_j)-F_a(u)| \le \norm{g_{\ep}}_{\infty} \int_{K} |F_a(u_j)-F_a(u)| \ra 0, \text{ as } j \ra \infty.
\end{align}
Next, we estimate the first integral of \eqref{compact1}. Since $(u_j)$ is bounded in $\Dtwo$, using  Remark \ref{classical embedding} we obtain
\begin{align*}
    & \intRn |g-g_{\ep}||F_a(u_j)-F_a(u)| \\
    & \le \intRn |g-g_{\ep}|\left(|F_a(u_j)| + |F_a(u)| \right)\\
    & \le C\norm{g-g_{\ep}}_{\frac{N}{4}} \left( \norm{\De u}^2_2 + \norm{\De u_j}^2_2 \right) + C \norm{g-g_{\ep}}_{\tilde{p}} \left( \norm{\De u}_2^{\ga} + \norm{\De u_j}_2^{\ga} \right) \\
    & \quad + aC \left( \norm{g-g_{\ep}}_1 \norm{g-g_{\ep}}_{\frac{N}{4}} \right)^{\frac{1}{2}} \left( \norm{\De u}_2 + \norm{\De u_j}_2 \right) \\ 
    & <C\ep,
\end{align*}
where $C$ is independent of $j.$ Therefore, from \eqref{compact1} and \eqref{compact2}, it follows  that  $N_a(u_j) \ra N_a(u)$, as $j \ra \infty$. 

 $(iii)$ Let $u_j \rightharpoonup u$ in $\D^{2,2}(\RN)$. We write 
\begin{align}\label{compact3}
     \intRn g \abs{F_{a_j}(u_j)- F_a(u)} \le \intRn g \abs{F_{a_j}(u_j)- F_a(u_j)} + \intRn g \abs{F_a(u_j)-F_{a}(u)}.
\end{align}
From the compactness of $N_a$, second integral of \eqref{compact3} converges to zero. For $g_{\ep}$ as given in the proof of $(ii)$, we write
\begin{equation}\label{compact4}
    \begin{split}
        \intRn g \abs{F_{a_j}(u_j)- F_a(u_j)} \le \intRn \abs{g-g_{\ep}}|F_{a_j}(u_j)-F_a(u_j)| \\ + \intRn |g_{\ep}||F_{a_j}(u_j)-F_a(u_j)|.
    \end{split}
\end{equation}
Using $a_j \ra a$ and the similar set of arguments as given in $(ii)$, we get $\int_{\RN} \abs{g-g_{\ep}}|F_{a_j}(u_j)-F_a(u_j)| < C\ep$ for some $C$ independent of $j.$ We require to estimate the second integral of \eqref{compact4}. Notice that $F_{a_j}(u_j)-F_a(u_j) = (a-a_j)u_j$. Therefore, since $(u_j)$ is bounded in $\Dtwo$,
\begin{align*}
   \int_{K} |F_{a_j}(u_j)-F_a(u_j)| = |a-a_j|\int_{K} |u_j| \le |a-a_j||K|^{\frac{1}{2}} \left( \int_{K} u_j^2 \right)^{\frac{1}{2}} \le C|a-a_j|,
\end{align*}
where $K$ is the support of $g_{\ep}$, and hence $\int_{K} |F_{a_j}(u_j)-F_a(u_j)| \rightarrow 0$ as $j \ra \infty$. Thus, combining \eqref{compact3} and \eqref{compact4}, we get $\int_{\RN} g F_{a_j}(u_j) \ra \int_{\RN} g F_a(u)$.
\end{proof}

Now we define the Cerami condition for a functional on $\Dtwo$, introduced in \cite{Cerami1978}. 
 
\begin{definition}
Let $J \in C^1(\D^{2,2}(\RN), \R)$ and let $c \in \R$. The functional $J$  is said to satisfy the Cerami condition ($C$-condition) at level $c$, if the following holds:

\begin{enumerate}[label={($\bf C_{\arabic*}$)}]
\setcounter{enumi}{0}
    \item \label{Cerami1} For any bounded sequence $(u_j)$ in $\Dtwo$ such that $J(u_j) \rightarrow c$ in $\R$ and $J'(u_j) \rightarrow 0$ in $(\D^{2,2}(\RN))'$, there exists a convergent subsequence of $(u_j)$.  \vspace{0.1 cm}
    \item \label{Cerami2}  There exist $\eta, \be, \rho >0$ such that 
\begin{align*}
    \norm{J'(u)}\norm{ \De u}_2 \ge \be, \; \forall \, u \in J^{-1} \left([c-\eta,c+\eta] \right) \text{ with } \norm{\De u}_2 \ge \rho.
\end{align*}
\end{enumerate}
\end{definition}

In the following proposition we show that $I_a$ satisfies \ref{Cerami1}.

\begin{proposition}\label{CC1}
Let $f$ satisfies {\rm \ref{f1}}, and let $g$ be positive and satisfies {\rm \ref{g1}}. Let $(u_j)$ be a bounded sequence in $\Dtwo$ such that $I_a(u_j) \rightarrow c \in \R$ and $I_a'(u_j) \rightarrow 0$ in $(\D^{2,2}(\RN))'$. Then $(u_j)$ has a convergent subsequence in $\Dtwo$.  
\end{proposition}
\begin{proof}
By the reflexivity, up to a subsequence $u_j \rightharpoonup u$ in $\Dtwo$. We consider the functional $J(v) = \int_{\RN} (\De v)^2$, for $v \in \Dtwo$. Observe that $J \in C^1(\Dtwo, \R)$ and
\begin{align}\label{PS6}
    \left< J'(u_j), u_j-u \right> = \left<I_a'(u_j), u_j-u \right> + \left<N'_a(u_j), u_j-u \right>.
\end{align}
By the hypothesis, $\abs{\left<I_a'(u_j), u_j-u \right>} \le \norm{I_a'(u_j)} \norm{u_j-u} \ra 0$ as $j \ra \infty$. Now we show that $\left<N'_a(u_j), u_j-u \right> \ra 0$ as $j \ra \infty$. Using Remark \ref{bound}-$(i)$, 
\begin{equation}\label{PS7}
    \begin{split}
        \left| \left<N'_a(u_j), u_j-u \right>  \right| & \le \intRn g|f_a(u_j)| |u_j-u|   \\ & \le C \intRn g\left(|u_j| + |u_j|^{\ga -1} + 1\right)|u_j-u|.
    \end{split}
\end{equation}
Using the compact embeddings of $\Dtwo \hookrightarrow L^{\tilde{\ga}}(g, \RN)$ for $\tilde{\ga} \in [2,2^{**})$ (Proposition \ref{compact embeddings}), we get 
\begin{align*}
    &\intRn g |u_j| |u_j-u| \le \left( \intRn g u_j^2 \right)^{\frac{1}{2}} \left( \intRn g |u_j-u|^2 \right)^{\frac{1}{2}} \ra 0,  \\
    & \intRn g |u_j|^{\ga -1} |u_j-u| \le \left( \intRn g |u_j|^{\ga} \right)^{\frac{1}{\ga'}} \left( \intRn g |u_j-u|^{\ga} \right)^{\frac{1}{\ga}} \ra 0, \\
    & \intRn g|u_j-u| \le \norm{g}_1^{\frac{1}{2}} \left( \intRn g |u_j-u|^2 \right)^{\frac{1}{2}} \ra 0,
\end{align*}
as $j \ra \infty$. Therefore, from \eqref{PS7}, $\left<N'_a(u_j), u_j-u \right> \ra 0$. Thus using \eqref{PS6}, we obtain $\int_{\RN} \De u_j \, \De (u_j - u) = \left< J'(u_j), u_j-u \right> \ra 0.$ Further, since $u_j \rightharpoonup u$ in $\Dtwo$, we also have $\int_{\RN} \De u \, \De(u_j-u) \ra 0.$ Therefore, $\int_{\RN} (\De(u_j-u))^2 \ra 0$, as required.
\end{proof}

Now we prove that $I_a$ satisfies \ref{Cerami2}.
\begin{proposition}\label{bounded}
Let $f$ satisfies {\rm \ref{f1}-\ref{f3}}, and let $g$ be positive and satisfies {\rm \ref{g1}}. Then for any $c\in \R$, there exist $\eta, \be, \rho >0$ such that $\norm{I_a'(u)}\norm{\De u}_2 \ge \be$ for $u \in I_a^{-1} \left([c-\eta,c+\eta] \right)$ with $\norm{\De u}_2 \ge \rho$.
\end{proposition} 

\begin{proof}
Our proof uses the method of contradiction. Suppose $(u_j)$ is a sequence in $\D^{2,2}(\RN)$ such that
\begin{align*}
    I_a(u_j) \ra c, \norm{\De u_j}_2 \ra \infty, \text{ and } \norm{I_a'(u_j)}\norm{\De u_j}_2 \ra 0, \text{ as } j \ra \infty.
\end{align*}
We set $w_j= \frac{u_j}{\norm{\De u_j}_2}$. Then $\norm{\De w_j}_2=1$. By the reflexivity, up to a subsequece, $w_j \rightharpoonup w$ in $\Dtwo$. We divide our proof into two steps. \\
\noi  \textbf{\underline{Step 1:}} We consider the set $\Om= \left\{x \in \RN : w(x)>0 \right\}$. In this step we show that $|\Om|=0$. On a contrary, we assume $|\Om|>0$. Using $\left<I_a'(u_j), u_j \right> \le \norm{I_a'(u_j)}\norm{\De u_j}_2 \ra 0$, we write
\begin{align*}
    \norm{\De u_j}^2_2 - \intRn gf_a(u_j)u_j= \ep_j, \text{ where } \ep_j \ra 0, \text{ as } j \ra \infty.
\end{align*}
From the above identity, we get
\begin{align}\label{measure01}
    1 =  \frac{1}{\norm{\De u_j}^2_2} \left\{ \intRn gf_a(u_j)u_j + \ep_j \right\}.
\end{align}
Let $\Om^+_j= \left\{ x \in \RN : u_j(x) \ge 0 \right\}.$ From the compactness of $\Dtwo \hookrightarrow L^2_{loc}(\RN)$, we infer that $w_j(x) \ra w(x)$ a.e. in $\RN$ (up to a subsequence). In particular, for any $\Om_0 \subset \Om$ with $0<|\Om_0|<\infty$, $w_j(x) \ra w(x)$ a.e. in $\Om_0$. Hence, by the Egorov's theorem, there exists $\Om_1 \subset \Om_0$ with $|\Om_1|>0$ such that $(w_j)$ converges to $w$ uniformly on $\Om_1$. Thus there exists $j_1 \in \N$ such that for $j \ge j_1$, $w_j \ge 0$ and hence $u_j \ge 0$ a.e. on $\Om_1$. Therefore, $\Om_1 \subset \Om^+_j$, for all \,$ j \ge j_1$. Further, $f_a(u_j)u_j \ge 0$ a.e. on $\RN \setminus \Om^+_j$, and using Remark \ref{bound}-$(ii)$, for any $M>0$ there exists $C_M>0$ such that $f_a(u_j)\ge Mu_j-(C_M+a)$ a.e. on $\Om^+_j$. Therefore, from \eqref{measure01}, for all $j \ge j_1$, we obtain  
\begin{equation}\label{measure02}
    \begin{split}
        1  & \ge \frac{1}{\norm{\De u_j}^2_2} \left\{ \int_{\Om^+_j} gf_a(u_j)u_j + \ep_j \right\} \\ 
    & \ge M \int_{\Om^+_j} g \frac{u_j^2}{\norm{\De u_j}^2_2} -  \frac{(C_M + a)}{\norm{\De u_j}_2} \int_{\Om^+_j} g \frac{u_j}{\norm{\De u_j}_2} + \frac{\ep_j}{\norm{\De u_j}^2_2} \\
    & \ge M \int_{\Om_1} g w_j^2 -  \frac{(C_M + a)}{\norm{\De u_j}_2} \int_{\Om^+_j} g w_j + \frac{\ep_j}{\norm{\De u_j}^2_2}.
    \end{split}
\end{equation}
From the compactness of $G$ (Proposition \ref{compact map}-$(i)$), it follows that $\int_{\Om_1} gw_j^2 \ra \int_{\Om_1} gw^2$. Moreover, using the H\"{o}lder's inequality and \eqref{cembed2}, 
$$\int_{\Om^+_j} g w_j  \le \norm{g}^{\frac{1}{2}}_1 \left( \int_{\RN} gw_j^2 \right)^{\frac{1}{2}} \le C \left( \norm{g}_1 \norm{g}_{\frac{N}{4}} \right)^{\frac{1}{2}},$$ 
where $C=C(N)$. Now taking the limit as $j \ra \infty$ in \eqref{measure02}, 
\begin{align*}
    1 \ge M \int_{\Om_1} gw^2, \text{ for sufficiently large } M>0,
\end{align*}
a contradiction. Thus, $|\Om|=0$. \\
\noi  \textbf{\underline{Step 2:}} For a fixed $j \in \N$, we set $m_j= \sup \left\{ I_a(tw_j): 0 \le t \le \norm{\De u_j}_2 \right\}$. Since the map $t \mapsto I_a(tw_j)$ is continuous on $[0,\norm{\De u_j}_2]$, there exists $t_j \in [0,\norm{\De u_j}_2]$ such that $m_j=I_a(t_jw_j)$. First, we show that $m_j \ra \infty,$ as $j \ra \infty$. Since the sequence $(u_j)$ is unbounded, there exists $j_2 \in \N$ so that for $j \ge j_2$, $\norm{\De u_j}_2 \ge M$, and hence $m_j \ge I_a(Mw_j)$. Using the compactness of $N_a$ (Proposition \ref{compact map}-$(ii)$) and $|\Om| = 0$ (Step 1), we get 
\begin{align*}
    \lim_{j \ra \infty} I_a(Mw_j) & = \frac{M^2}{2} - \lim_{j \ra \infty} \intRn gF_a(Mw_j) \\
    & = \frac{M^2}{2} - \intRn gF_a(Mw) \\
    & = \frac{M^2}{2} + aM\int_{\Om^c} gw.
\end{align*}
Notice that, the term $\frac{M^2}{2} + aM\int_{\Om^c} gw$ is sufficiently large. Therefore,
\begin{align}\label{unbounded1}
    \lim_{j \ra \infty} m_j = \infty. 
\end{align}
Next, for each $j \in \N$ we calculate
\begin{align}\label{PS4}
    I_a(t_jw_j)-I_a(u_j) = \frac{t_j^2-\norm{\De u_j}^2_2}{2} + \intRn g \left(F_a(u_j) - F_a(t_jw_j) \right). 
\end{align}
Set $s_j=\frac{t_j}{\norm{\De u_j}_2}$. Then $s_j \in [0,1]$ and $s_ju_j=t_jw_j$. It is easy to observe that  $0<s_ju_j(x) \le u_j(x)$, if $u_j(x) > 0$ and $0>s_ju_j(x) \ge u_j(x)$, if $u_j(x) < 0$. Hence using Proposition \ref{relation}, if $u_j(x) \neq 0$, then 
\begin{align*}
    F_a(u_j(x)) - F_a(s_ju_j(x)) & \le \frac{(u_j(x))^2-s_j^2(u_j(x))^2}{2u_j(x)}f_a(u_j(x)) +C_R \\ & \le \frac{1-s_j^2}{2} u_j(x)f_a(u_j(x)) +C_R.
\end{align*}
Further, if $u_j(x)=0$, then $F_a(u_j(x)) - F_a(s_ju_j(x)) =0$. Thus the above inequality yields
\begin{equation}\label{PS5}
    \begin{split}
         \intRn g(x) \big(F_a(u_j(x)) - F_a(s_ju_j(x)) \big) \, \dx  \le \intRn g(x) \\
        \left( \frac{1-s_j^2}{2} u_j(x)f_a(u_j(x)) +C_R \right) \, \dx. 
    \end{split}
\end{equation}
Therefore, \eqref{PS4} and \eqref{PS5} yield
\begin{align*}
    I_a(t_jw_j)-I_a(u_j) & \le \frac{1-s_j^2}{2} \left( -\norm{\De u_j}^2_2 + \intRn gu_jf_a(u_j) \right) + C_R \norm{g}_1 \\
    & = \frac{s_j^2-1}{2} \left< I_a'(u_j), u_j \right> + C_R \norm{g}_1.
\end{align*}
Now, since $\left< I_a'(u_j), u_j \right> \ra 0$, it follows that $(I_a(t_jw_j)-I_a(u_j))$ is a bounded sequence. On the other hand, using \eqref{unbounded1} and the fact that $(I_a(u_j))$ is bounded (as $\lim_{j \ra \infty}I_a(u_j) = c$), we conclude that $I_a(t_jw_j)-I_a(u_j) \ra \infty$, as $j \ra \infty$, resulting in a contradiction. Therefore, such a unbounded sequence $(u_j)$ does not exist. 
\end{proof}

Next we prove the mountain pass geometry for $I_a$. For $\rho>0$, we consider the following set
\begin{align*}
    \S_{\rho}= \left\{u: \Dtwo : \norm{\De u}_2 = \rho \right\}.
\end{align*}

\begin{lemma}\label{MPG}
Let $f$ satisfies {\rm \ref{f1}-\ref{f2}}, and let $g$ be positive and satisfies {\rm \ref{g1}}. Then the following holds: 
\begin{enumerate}
    \item[(i)] There exist $\rho>0, \be= \be(\rho)>0$, and $a_1=a_1(\rho)>0$ such that if $a \in (0, a_1)$, then $I_a(u) \ge \be$ for every $u\in \S_{\rho}$.
    \item[(ii)] There exists $\tilde{v} \in \Dtwo$ with $\norm{\De \tilde{v}}_2 > \rho$ such that $I_a(\tilde{v})<0$.
\end{enumerate}
\end{lemma}

\begin{proof}
$(i)$ Let $\rho>0$ and $u \in \S_{\rho}$. Choose $\ep < (\norm{g}_{\frac{N}{4}}B_g)^{-1}$ where $B_g$ is the best constant of \eqref{cembed2} for $\tilde{\ga}=2$. Then using $F_a(u) \le \frac{\ep}{2} u^2 + \frac{C}{\ga} |u|^{\ga} + a|u|$ (Remark \ref{bound}-$(i)$) and \eqref{cembed2}, we write
\begin{equation}\label{MP1}
    \begin{split}
    I_a(u) & \ge \frac{\norm{\De u}^2_2}{2} - \frac{\ep}{2} \intRn gu^2 - \frac{C}{\ga}  \intRn g|u|^{\ga} - a \intRn g|u|  \\
     & \ge \norm{\De u}^2_2 \left\{\frac{1}{2} - \frac{\ep}{2} B_g \norm{g}_{\frac{N}{4}} - C \norm{\De u}^{\ga-2}_2 \right\} - aC\norm{\De u}_2  \\
     & = A(\rho) - aC\rho,
    \end{split}
\end{equation}
where $A(\rho)= \rho^2 (\frac{1}{2} - \frac{\ep}{2} B_g \norm{g}_{\frac{N}{4}} - C \rho^{\ga -2})$ and $C$ is independent of $a$. Let $\rho_1$ be the first nontrivial zero of $A$. Choose 
\begin{align*}
    0< a_1< \frac{A(\rho)}{C\rho} \text{ and } \be=A(\rho) - a_1C\rho, \text{ where } \rho< \rho_1.
\end{align*}
Therefore, from \eqref{MP1}, $I_a(u) \ge \be$ for all $a < a_1$. 

 $(ii)$ We consider $\phi \in \C^2(\RN), \phi \ge 0$, and $\norm{\De \phi}_2=1$. For any $M>0$, using Remark \ref{bound}-$(ii)$, $F_a(t \phi) = F(t\phi)-at\phi \ge M(t\phi)^2-(C_M + at\phi)$, for all $t>0$. Hence 
\begin{align*}
    I_a(t \phi) \le t^2 \left( \frac{1}{2} - M \intRn g \phi^2 \right) + at \intRn g \phi + C_M \norm{g}_1.
\end{align*}
For $M>\left( 2\int_{\RN} g \phi^2 \right)^{-1}$, $I_a(t \phi) \ra -\infty$, as $t \ra \infty$, i.e., there exists $t_1>\rho$ so that $I_a(t \phi)<0$ for $t > t_1$. Thus $\tilde{v} = t\phi$ with $t>t_1$ is the required function. 
\end{proof}

\section{Proof of the main theorem}\label{sec3}
The first part of this section is devoted to discussing the existence and various qualitative properties of the solutions for \eqref{SP}. In the second part, we prove the positivity of the solutions.

\subsection{Existence and qualitative properties of solutions}
This subsection considers the existence, regularity, and uniform boundedness of the mountain pass solutions of \eqref{SP}. For the existence of solutions, we use the following theorem due to Schechter in \cite[Theorem 2.1]{S1991}.

\begin{theorem}[Mountain-Pass Theorem]\label{MPT}
Let $J \in C^1(\D^{2,2}(\RN), \R)$. We assume that
\begin{enumerate}
    \item[(i)]  $J(0)=0$ and $J$ satisfies the $C$-condition. 
    \item[(ii)] There exist $\rho>0$ and $\be>0$ such that $\inf \left\{J(u) : u \in \S_{\rho} \right\} \ge \be$.
    \item[(iii)] There exists $v \in \Dtwo$ with $\norm{\De v}_2 > \rho$ such that $J(v) < 0$. 
\end{enumerate}
Let $\Ga_v := \left\{ \ga \in C([0,1], \Dtwo) : \ga(0) = 0 \text{ and } \ga(1) = v  \right\}$. 
Then $$c_v:= \inf_{\ga \in \Ga_v} \max_{s \in [0,1]} J(\ga(s)),$$ 
is a critical value of $J$. Further, $c_v \ge \inf_{u \in \S_{\rho}} J(u)$.
\end{theorem}

\begin{theorem}\label{Existence}
 Let $f$ satisfies {\rm \ref{f1}-\ref{f3}}. Let $g$ be positive and satisfies {\rm \ref{g1}}. Then there exists $a_1>0$ such that for each $a \in (0,a_1)$, \eqref{SP} admits a nontrivial solution. 
\end{theorem}

\begin{proof}
Let $a_1, \beta, \tilde{v}$ be as given in Lemma \ref{MPG}. Then for $a \in (0, a_1)$, using Proposition \ref{CC1}-\ref{bounded} and Lemma \ref{MPG}, $I_a$ satisfies all the properties of the Mountain-Pass theorem. Therefore, by Theorem \ref{MPT}, there exists $u_a \in \Dtwo$ satisfying
\begin{align}\label{MP2} 
    I_a(u_a) = c_{a,\tilde{v}}:=\inf_{\ga \in \Ga_{\tilde{v}}} \max_{s \in [0,1]} I_a(\ga(s)) \ge \be, \text{ and } I_a'(u_a)=0.
\end{align}
Thus, $u_a$ is a nontrivial solution of \eqref{SP}.  
\end{proof}
 
\begin{remark}
Let $f$ satisfies the following (RA) type nonlinearities:
\begin{align*}
    \text{there exist } \theta > 2 \text{ and } t_0>0 \text{ such that }
    0<\theta F(t)\le tf(t), \; \forall \, t \ge t_0.
\end{align*}
Then there exist $C_1,C_2 >0$ such that $F(t) \ge C_1t^{\theta} - C_2$ and $f(t) \ge C_1 \theta t^{\theta-1}$, for $t \in \R^+$. Therefore, $\lim_{t \ra \infty} \frac{f(t)}{t}= \infty$ (which is \ref{f2}). Thus, if $f$ satisfies \ref{f1} and $g$ satisfies \ref{g1}, then by Theorem \ref{Existence}, for each $a \in (0,a_1)$, \eqref{SP} admits a mountain pass solution $u_a$.
\end{remark}

\begin{definition}
We define the following set in $ \Dtwo$:
\begin{align*}
    \M = \left\{ v \in \Dtwo: v \text{ is a mountain pass solution for } \eqref{SP}  \right\}.
\end{align*}
\end{definition}
From Theorem \ref{Existence}, $u_a \in \M$ for every $a \in (0,a_1)$. In the rest of this section we study the properties of functions in $\M$. Now we show that $u_a$ is continuous and locally bounded, using the following regularity result \cite[Theorem 1.1]{Medereski} by Mederski and Siemianowski. 

\begin{theorem}[Regularity Theorem]\label{Mederski}
Let $N \ge  5$, $h \in L_{loc}^{\frac{N}{4}}(\RN)$ be non-negative, and $\tilde{f}$ be a Carath\'{e}odory function that satisfies $|\tilde{f}(x,s)| \le h(x) (1+|s|)$ for $s \in \R$ and a.e. $x \in \RN$. Let $u \in W_{loc}^{2,2}(\RN)$ be a weak solution of $\De^2 u = \tilde{f}(x,u) \text{ in } \RN$. Then  $u \in W_{loc}^{4,q}(\RN)$ for any $q \in [1, \infty)$.
\end{theorem}

\begin{proposition}\label{Regularity2}
Let $f,g,a_1$ be as given in Theorem \ref{Existence}. Let $u_a \in \M$, for all $a \in (0,a_1)$. Then $u_a \in W_{loc}^{4,q}(\RN)$ where $q \in [1, \infty)$.
\end{proposition}

\begin{proof}
For $ a \in (0,a_1)$, we consider the following function: $$h = \frac{g \left( 1+|u_a|^{\ga-1} \right)}{1+ |u_a|}, \text{ where } \ga \in (2,2^{**}).$$
Observe that $h \le g (1+ |u_a|^{\ga -2})$. We set $r= \frac{8}{(N-4)(\ga-2)}$. For $\ga \in (2,2^{**})$ we have $r>1$.  Using the H\"{o}lder's inequality with the conjugate pair $(r,r')$ and the embedding $\Dtwo \hookrightarrow L^{2^{**}}(\RN)$, we get 
\begin{equation*}
    \begin{split}
        \intRn h^{\frac{N}{4}} & \le C \intRn g^{\frac{N}{4}} \left( 1 + |u_a|^{\frac{(\ga-2)N}{4}} \right) \\
        & \le C \left( \intRn g^{\frac{N}{4}} + \left( \intRn g^{\frac{N}{4}r'} \right)^{\frac{1}{r'}} \left( \intRn |u_a|^{2^{**}} \right)^{\frac{1}{r}} \right) \\
        & \le C \left(\norm{g}_{\frac{N}{4}}^{\frac{N}{4}}+ \norm{g}_{\frac{N r'}{4}}^{\frac{N}{4}}\left( \intRn (\De u_a)^2 \right)^{\frac{N}{(N-4)r}} \right),
    \end{split}
\end{equation*}
for some $C > 0$. Therefore, $h \in L^{\frac{N}{4}}(\RN)$. Further, $|f_a(u_a)| \le C(1+|u_a|^{\gamma -1})$ where $C=C(C_f,t_2,a_1)$ (using Remark \ref{bound}-$(i)$). Hence $g|f_a(u_a)| \le C h (1+ |u_a|).$
Moreover, $u_a \in W_{loc}^{2,2}(\RN)$ (by Proposition \ref{local}-$(ii)$). Thus, by taking $\tilde{f}(x,s)= g(|x|)f_a(s)$ we see that all the hypothesis of Theorem \ref{Mederski} are satisfied. Therefore, by Theorem \ref{Mederski}, $u_a \in W_{loc}^{4,q}(\RN)$ for any $q \in [1, \infty)$. 
\end{proof}

\begin{remark}\label{bounded1.1}
Let $a \in (0,a_1)$. Then for $q>\frac{N}{4}$, using Proposition \ref{Regularity2} and the Sobolev embedding we get $u_a \in C_{loc}^{3, \al}(\RN)$ where $\al \in (0,1)$. Hence it is evident that $u_a \in L_{loc}^{\infty}(\RN) \cap C(\RN)$.
\end{remark}

Next, we discuss the uniform boundedness of $(u_a)$. First, we prove that $\{u_a, a \in (0,a_2)\}$ (for some $a_2 \in (0,a_1)$) are uniformly bounded in $\Dtwo$. For that we require the following lemma.

\begin{lemma}\label{bound1}
Let $f, g,a_1$ be as given in Theorem \ref{Existence}. Let $u_a \in \M$, for all $a \in (0,a_1)$. Then there exists $C>0$ such that $I_a(u_a) \le C$ for all $a \in (0,a_1)$.
\end{lemma}

\begin{proof}
For $\phi, t_1, \tilde{v}$ as given in the proof of Lemma \ref{MPG}, we define $\tilde{\ga}: [0,1] \ra \Dtwo$ by $\tilde{\ga}(s) = s\tilde{v}$. Recall that $\tilde{v}=t\phi$ for some $t>t_1$. Clearly, $\tilde{\ga} \in \Gamma_{\tilde{v}}$ and hence using Theorem \ref{Existence}, for $a \in (0,a_1)$,
\begin{align*}
    I_a(u_a) \le \max_{s \in [0,1]} I_a(\tilde{\ga}(s)) = \max_{s \in [0,1]} I_a(st\phi).
\end{align*}
Further, using Remark \ref{bound}-$(ii)$, $F_a(st\phi) = F(st\phi)- ast\phi \ge M(st\phi)^2-C_M-a_1st\phi$. Using this estimate we get
\begin{align*}
    \max_{s \in [0,1]} I_a(st\phi) \le \max_{s \in [0,1]} \left\{ s^2 t^2 \left( \frac{1}{2} - M \int_{\RN} g \phi^2 \right)  + st a_1 \int_{\RN} g \phi + C_M \norm{g}_1\right\} \le C.
\end{align*}
Therefore, $I_a(u_a) \le C$ for all $a \in (0, a_1)$.
\end{proof}

\begin{proposition}\label{bound solution}
Let $f, g,a_1,u_a$ be as given in Theorem \ref{Existence}. Then there exist $C>0$ and $a_2 \in (0,a_1)$ such that $\norm{\De u_a}_2 \le C$ for all $a \in (0, a_2)$. 
\end{proposition}

\begin{proof}
Our proof uses the method of contraction. On a contrary, assume that there no such $a_2$ exist. Then there exists a sequence $(a_j)$ in $(0,a_1)$, such that $ a_j \ra 0,$ and $\norm{\De u_{a_j}}_2 \ra \infty,$ as $j \ra \infty.$ Set $w_j= u_{a_j}\norm{\De u_{a_j}}^{-1}_2$. Then $w_j \rightharpoonup w$ in $\Dtwo$. 
For each $j \in \N$, since $ \left< I_{a_j}'(u_{a_j}), u_{a_j} \right>= 0$, we have 
\begin{align*}
   \norm{\De u_{a_j}}^2_2 = \intRn g f_{a_j}(u_{a_j})u_{a_j}. 
\end{align*}
Now following the same arguments as given in Step 1 of Proposition \ref{bounded}, we can obtain $\abs{\{w>0\}} = 0$. 
Set $m_{a_j}= \max \left\{ I_{a_j}(tw_j): 0 \le t \le \norm{\De u_{a_j}}_2 \right\}$. Then $m_{a_j} = I_{a_j}(t_{a_j} w_j)$ for some $ 0 \le t_{a_j} \le \norm{\De u_{a_j}}_2$. Since $(u_{a_j})$ is unbounded in $\Dtwo$, for any $M>0$ there exists $j_1 \in \N$ so that $\norm{\De u_{a_j}}_2 \ge M$ and $m_{a_j} \ge I_{a_j}(Mw_j)$ for $j \ge j_1$. Further, using Proposition \ref{compact map}-$(iii)$ and $\abs{\{w>0\}} = 0$, we get
\begin{align*}
    \lim_{j \ra \infty} I_{a_j}(Mw_j) = \frac{M^2}{2} - \lim_{j \ra \infty}\intRn gF_{a_j}(Mw_j) =  \frac{M^2}{2} - \intRn gF_0(Mw) = \frac{M^2}{2}.
\end{align*}
From the above identity it follows that $I_{a_j}(t_{a_j} w_j) \ra \infty$, and hence using Lemma \ref{bound1}, $I_{a_j}(t_{a_j} w_j) - I_{a_j}(u_{a_j}) \ra \infty$, as $j \ra \infty$ . On the other hand, for $s_j=t_{a_j} \norm{\De u_{a_j}}^{-1}_2$, using Proposition \ref{relation} and \eqref{PS5} with $a=a_j$, we obtain
\begin{align*}
    I_{a_j}(t_{a_j}w_j) - I_{a_j}(u_{a_j}) 
    & \le  \frac{s_j^2-1}{2} \left< I_{a_j}'(u_{a_j}), u_{a_j} \right> +C_R \norm{g}_1  =C_R\norm{g}_1,
    \end{align*}
a contradiction. Thus there must exists $a_2 \in (0,a_1)$ such that $\norm{\De u_a}_2 \le C$ for all $a \in (0, a_2)$.  
\end{proof}

In the following proposition, we prove that $\{u_a, a \in (0,a_2)\} \subset L^{\infty}(\RN)$ and $\{u_a\big|_{B_1^c}, a \in (0,a_2)\}$ is uniformly bounded. 

\begin{proposition}\label{Regularity1}
Let $f, g$ be as given in Theorem \ref{Existence}, and $a_2$ be as given in Proposition \ref{bound solution}.  Let $u_a \in \M$, for every $a \in (0,a_2)$. In addition, we assume that $g$ satisfies {\rm \ref{g2}}. Then $u_a \in L^{\infty}(\RN)$. Moreover, there exists $C>0$ such that $\norm{u_a}_{L^{\infty}(B_1^c)} \le C$, for all $a \in (0,a_2)$.
\end{proposition}

\begin{proof}
Let $a \in (0,a_2)$. Since $u_a$ is a weak solution of \eqref{SP}, using Remark \ref{Riesz0.1} and Proposition \ref{fundamental}, we write
\begin{align*}
    u_a(x) & = \mathcal{R}_4 \intRn \frac{g(y) f_a(u_a(y))}{|x-y|^{N-4}} \, \dy \; \text{ a.e. in } \RN.
\end{align*}
We split 
\begin{align}\label{Riesz1}
    \intRn \frac{g(y) f_a(u_a(y))}{|x-y|^{N-4}} \, \dy = \int_{B_1(x)} \frac{g(y) f_a(u_a(y))}{|x-y|^{N-4}} \, \dy + \int_{B_1^c(x)} \frac{g(y) f_a(u_a(y))}{|x-y|^{N-4}} \, \dy.
\end{align}
Taking $\de=2^{**}$, we use the H\"{o}lder's inequality with the conjugate pair $(\de,\de')$ to estimate the first integral of \eqref{Riesz1} as
\begin{equation*}
    \begin{split}
        \int_{B_1(x)} \frac{g(y) |f_a(u_a(y))|}{|x-y|^{N-4}} \, \dy \le \left( \int_{B_1(x)} \frac{g(y)^{\de}}{|x-y|^{(N-4)\de}}  \, \dy\right)^{\frac{1}{\de}} \\ \left( \int_{B_1(x)} \left( |f_a(u_a(y))| \right)^{\de'} \, \dy \right)^{\frac{1}{\de'}}.
    \end{split}
\end{equation*}
Now since $g$ satisfies \ref{g2}, 
\begin{align*}
  \int_{B_1(x)} \frac{g(y)^{\de}}{|x-y|^{(N-4)\de}}  \, \dy \le \intRn \frac{g(y)^{\de}}{|x-y|^{(N-4)\de}}  \, \dy \le C_g|x|^{(4-N)\de} \le C_g, \text{ for  } x \in B^c_1.
\end{align*}
From Remark \ref{bound}-$(i)$, we have $|f_a(u_a)| \le C (1+ |u_a|^{2^{**}-1})$ where $C=C(C_f,t_2,a_2)$. Hence using the embedding $\Dtwo \hookrightarrow L^{2^{**}}(\RN)$ and Proposition \ref{bound solution} we get
\begin{align*}
     \int_{B_1(x)} \left( |f_a(u_a(y))| \right)^{\de'} \, \dy & \le C2^{\de'-1} \int_{B_1(x)} \left( 1 + |u_a(y)|^{2^{**}} \right) \, \dy \\
     & \le C2^{\de'-1} \left(|B_1(x)| + \norm{u_a}_{2^{**}}^{2^{**}} \right) \\
     & \le C2^{\de'-1} \left( |B_1(0)| + \norm{\De u_a}_2^{2^{**}} \right) \\
     & \le C,
\end{align*}
for some $C$ which is independent of $a$. Next we estimate the second integral of \eqref{Riesz1}. Using Remark \ref{classical embedding}, $\Dtwo \hookrightarrow L^{2^{**}}(\RN)$, and Proposition \ref{bound solution},
 \begin{align*}
     \int_{B_1^c(x)} \frac{g(y) |f_a(u_a(y))|}{|x-y|^{N-4}} \, \dy & \le  \int_{B_1^c(x)} g(y) |f_a(u_a(y))| \, \dy \\
     & \le C \int_{\RN} g(y) \left( 1 +  |u_a(y)|^{2^{**}-1} \right) \, \dy \\
     & \le C \left(\norm{g}_1 + \norm{g}_1^{\frac{1}{2^{**}}} \norm{g}_{\infty}^{\frac{1}{(2^{**})'}} \norm{\De u_a}_2 \right) \\
     & \le C,
 \end{align*}
where $C$ does not depend on $a$. Therefore, $\abs{u_a(x)} \le C$ for all $x \in B_1^c$. Hence $\norm{u_a}_{L^{\infty}(B_1^c)} \le C$ for all $a\in (0,a_2)$. Further using Remark \ref{bounded1.1} we conclude that $u_a \in L^{\infty}(\RN)$.
\end{proof}

Now we establish a uniform lower bound for the mountain pass solutions in $L^{\infty}(\RN)$.

\begin{proposition}\label{lowerbound}
Let $f,g,a_2,u_a$ be as given in Proposition \ref{Regularity1}. Then there exist $\tilde{a}_2 \in (0,a_2)$ and $\be_1 > 0$ such that $\norm{u_a}_{\infty} \ge \be_1$, for all $a \in (0,\tilde{a}_2)$.
\end{proposition}

\begin{proof}
By the definition, $F_a(t) \ge -a|t|$, for all $t \in \R$. For $\beta$ as given in Lemma \ref{MPG} and using \eqref{MP2} we see that $I_a(u_a) \ge \be$, for all $a \in (0,a_2)$. Hence using Proposition \ref{bound solution}, we get for all $a < a_2$, 
\begin{align*}
    \frac{\norm{\De u_a}^2_2}{2} = I_a(u_a) + \intRn gF_a(u_a) \ge \be - a\intRn g|u_a| \ge \be - aC \left(\norm{g}_1 \norm{g}_{\frac{N}{4}} \right)^{\frac{1}{2}}.
\end{align*}
Choose $0 < \tilde{a}_2 < \min \left\{ \be C^{-1} \left(\norm{g}_1 \norm{g}_{\frac{N}{4}} \right)^{-\frac{1}{2}}, a_2 \right\}$. Then 
$$ \frac{\norm{\De u_a}^2_2}{2}  \ge \be_0 :=\be - \tilde{a}_2 C\left(\norm{g}_1 \norm{g}_{\frac{N}{4}}\right)^{\frac{1}{2}}>0, \, \forall \, a \in (0,\tilde{a}_2).$$ 
Hence using $|f_a(u_a)| \le C (1+ |u_a|^{2^{**}-1})$ (Remark \ref{bound}-$(i)$) and Proposition \ref{Regularity1}, we get
\begin{align*}
    \be_0 \le \frac{1}{2}\intRn g \abs{f_a(u_a)u_a} & \le C \intRn g \left( \abs{u_a} + |u_a|^{2^{**}} \right)  \le C \norm{g}_1 \left( \norm{u_a}_{\infty} + \norm{u_a}^{2^{**}}_{\infty} \right), 
\end{align*}
where $C$ does not depend on $a$. Therefore, there exists $\be_1 >0$ such that $\norm{u_a}_{\infty} \ge \be_1$, for all $ a \in (0,\tilde{a}_2)$.
\end{proof}

\subsection{Existence of positive solutions}
In this section, we prove the existence of positive solutions for \eqref{SP}. Let $(a_j)$ be a sequence in $\R^+$. Recall the following functional
\begin{align*}
  I_{a_j}(u) = \frac{1}{2} \intRn (\De u)^2 - N_{a_j}(u), \; \forall \, u \in \Dtwo.  
\end{align*}

\begin{definition}
A sequence $(u_j)$ in $\D^{2,2}(\RN)$ is called a Palais-Smale sequence (PS-sequence) for $I_{a_j}$, if $I_{a_j}(u_j) \rightarrow c \in \R$ and $I_{a_j}'(u_j) \rightarrow 0$ in $(\D^{2,2}(\RN))'$, as $j \ra \infty$.
\end{definition}

\begin{remark}\label{PS sequence}
Let $a_j \in (0,a_1)$.  By Theorem \ref{Existence}, $I_{a_j}'(u_{a_j}) = 0$, and from Lemma \ref{bound1} up to a subsequence $I_{a_j}(u_{a_j}) \ra c \in \R$. Hence $(u_{a_j})$ is a PS-sequence for $I_{a_j}$. Further, from Proposition \ref{bound solution}, the PS-sequence $(u_{a_j})$ is bounded in $\Dtwo$. 
\end{remark}

Now we are ready to obtain the positivity of solutions. Our proof broadly follows the idea used in \cite[Theorem 1.1]{AHS20}.

\begin{theorem}\label{positive}
Let $f$ satisfies {\rm \ref{f1}-\ref{f3}} and let $g$ be positive and satisfies {\rm \ref{g1}-\ref{g2}}. Let $u_a \in \M$ for all $a \in (0, \tilde{a}_2)$. Then there exists $a_3 \in (0,\tilde{a}_2)$ such that for each $a \in (0, a_3)$, $u_a \ge 0$ in $\RN$. Further, if $f$ satisfies {\rm \ref{f4}}, then $u_a >0$ a.e. on $\RN$. 
\end{theorem}

\begin{proof}
Let $(a_j)$ be a sequence in $(0,a_3)$ such that $a_j \ra 0$ as $j \ra \infty$. By Theorem \ref{Existence}, there exists mountain pass solution $u_{a_j} \in \Dtwo$ (we denote by $u_j$) for each $a_j$. It is enough to prove the existence of $\tilde{j} \in \N$ such that for each $j \ge \tilde{j}$, $u_j$ is strictly positive on $\RN$. From Remark \ref{PS sequence},  $(u_j)$ is a bounded PS-sequence in $\Dtwo$. By the reflexivity, $u_j \rightharpoonup u$ in $\Dtwo$. Now using the boundedness of $(a_j)$ and the similar set of arguments as given in Proposition \ref{CC1}, we obtain that $u_j \ra \tilde{u}$ in $\Dtwo$. Up to a subsequence $u_j(x) \ra \tilde{u}(x)$ a.e. in  $\RN$. Now we split the rest of our proof into three steps. In the first step, we prove that  $\tilde{u}$ is nonnegative and $(u_j)$ converges uniformly to $\tilde{u}$ except on a measure zero set in $\RN$. We obtain the non negativity of $u_j$ in Step 2. The last step provides the positivity of $u_j$ for large $j$. 

\noi  \textbf{\underline{Step 1:}} From Proposition \ref{Regularity1}, for each $j \in \N$, $u_j \in L^{\infty}(\RN)$. Set $r = (2^{**})'$. We denote $f_{a_j}$ by $f_j$. From Remark \ref{bound}-$(i)$, $f_j(u_j) \le C(1+|u_j|^{2^{**}-1})$. Hence using  Proposition \ref{bound solution}, we get
 \begin{align*}
    \intRn g^r |f_j(u_j)|^r \le C 2^{r-1} \intRn g^r\left( 1+ |u_j|^{2^{**}} \right) \le C,
\end{align*}
where $C$ does not depend on $j$. By the reflexivity of $L^r(\RN)$, let $g f_j(u_j) \rightharpoonup z$ in $L^r(\RN)$. We consider the following function:
\begin{align*}
    f_0(t) = \left\{\begin{array}{ll} 
            f(t) , & \text {if }  t \ge 0; \\ 
            0, & \text{if} \; t \le 0.  \\
             \end{array} \right.
\end{align*}
Notice that $f_j(t) \ra f_0(t)$, as $j \ra \infty$. Now 
\begin{align}\label{triangle}
    \abs{f_j(u_j)-f_0(\tilde{u})} \le \abs{f_j(u_j)-f_0(u_j)} + \abs{f_0(u_j)-f_0(\tilde{u})}.
\end{align}
Since $u_j(x) \ra \tilde{u}(x)$ a.e. in $\RN$ and $f_0 \in C(\R)$, we have $f_0(u_j(x)) \ra f_0(\tilde{u}(x))$ a.e. in $\RN$, as $j \ra \infty$. Moreover, since $\abs{f_j(u_j)-f_0(u_j)} \le a_j$, from \eqref{triangle} it follows that $f_j(u_j(x)) \ra f_0(\tilde{u}(x))$ a.e. in $\RN$, as $j \ra \infty$. Now applying the Mazur's lemma, we obtain $z= g(|x|) f_0(\tilde{u}(x))$ a.e. on $\RN$. Therefore, using $u_j \in \M$, we have the following identity for every $\phi \in \C_c^{\infty}(\RN)$:
\begin{align*}
    \intRn \De \tilde{u} \, \De \phi = \lim_{j \ra \infty} \intRn \De u_{j} \, \De \phi = \lim_{j \ra \infty} \intRn g f_j(u_j) \phi = \intRn z\phi = \intRn gf_0(\tilde{u}) \phi.
\end{align*}
Thus, by the density argument, $\tilde{u}$ is a weak solution of the following problem:
\begin{equation}\label{P2}
\begin{aligned} 
 \De^2 u =  g(x)f_0(u) \text{ in } \RN, \quad u \in \Dtwo.
 \end{aligned}
 \end{equation}
Therefore, using Remark \ref{Riesz0.1} and Proposition \ref{fundamental},
 \begin{align}\label{Riesz5}
     \tilde{u}(x) = \mathcal{R}_4 \intRn \frac{g(y)f_0(\tilde{u}(y))}{|x-y|^{N-4}} \, \dy \ge 0 \; \text{ a.e. in } \RN.
 \end{align}
Further, using the similar set of arguments as given in Proposition \ref{Regularity1} and Proposition \ref{Regularity2}, we get $\tilde{u} \in L^{\infty}(\RN) \cap C(\RN)$. Moreover, since $u_j \in \M$, we have 
\begin{align*}
    u_j(x) = \mathcal{R}_4 \intRn \frac{g(y)f_j(u_j(y))}{|x-y|^{N-4}} \, \dy  \text{ a.e. in } \RN.
\end{align*}
Therefore,
\begin{align}\label{Riesz3}
    \abs{u_j(x) - \tilde{u}(x)} \le \mathcal{R}_4 \intRn \frac{g(y)\abs{f_j(u_j(y))-f_0(\tilde{u}(y))}}{|x-y|^{N-4}}  \, \dy. 
\end{align}
For every $x\in B_1^c$, H\"{o}lder's inequality with the conjugate pair $(\de,\de')$ (where $\de =2^{**}$) and the condition \ref{g2} yield
\begin{align*}
    & \int_{B_1(x)} \frac{g(y) \abs{f_j(u_j(y))-f_0(\tilde{u}(y))}}{|x-y|^{N-4}} \, \dy \\
    & \le \left( \int_{B_1(x)} \frac{g(y)^{\de}}{|x-y|^{(N-4)\de}}  \, \dy\right)^{\frac{1}{\de}} \left( \int_{B_1(x)} \left( \abs{f_j(u_j(y))-f_0(\tilde{u}(y))} \right)^{\de'} \, \dy \right)^{\frac{1}{\de'}} \\
    & \le C_g^{\frac{1}{\de}} \left( \int_{B_1(x)} \left( \abs{f_j(u_j(y))-f_0(\tilde{u}(y))} \right)^{\de'} \, \dy \right)^{\frac{1}{\de'}}.
\end{align*}
Moreover, 
\begin{align}\label{gdominated}
     \abs{f_j(u_j(y))-f_0(\tilde{u}(y))}^{\de'} \le  2^{\de'-1} \left( a_j^{\de'} + \abs{f_0(u_j)-f_0(\tilde{u})}^{\de'} \right).
\end{align}
Now we show $\int_{B_1(x)} \abs{f_0(u_j)-f_0(\tilde{u})}^{\de'} \ra 0$ as $j\ra \infty$. Observe that $f_0(t) \le C\left(1+|t|^{2^{**}-1} \right)$ for all $t\in \R$ (from Remark \ref{bound}-$(i)$), where $C=C(C_f,t_2)$. Hence
\begin{align*}
    \abs{f_0(u_j)-f_0(\tilde{u})}^{\de'} \le 2^{\de'-1}\left( |f_0(u_j)|^{\de'} + |f_0(\tilde{u})|^{\de'} \right) & \le  C 2^{2(\de'-1)} \left(2+  |u_j|^{2^{**}} + |\tilde{u}|^{2^{**}} \right). 
\end{align*}
Further, using $\Dtwo \hookrightarrow L^{2^{**}}(\RN)$ we get
\begin{align*}
   \lim_{j \ra \infty} \int_{B_1(x)} \left(2+  |u_j(y)|^{2^{**}} + |\tilde{u}(y)|^{2^{**}} \right) \, \dy = 2 \int_{B_1(x)} \left( 1 + |\tilde{u}(y)|^{2^{**}}  \right) \, \dy.
\end{align*}
Therefore, by the generalized dominated convergence theorem, $$\lim_{j \ra \infty} \int_{B_1(x)} \abs{f_0(u_j(y))-f_0(\tilde{u}(y))}^{\de'} \, \dy = 0.$$
Now from \eqref{gdominated} and again using the generalized dominated convergence theorem, 
$$\lim_{j \ra \infty} \int_{B_1(x)} \abs{f_j(u_j(y))-f_0(\tilde{u}(y))}^{\de'} \, \dy = 0.$$ 
Using the similar estimate as given in the proof of Proposition \ref{Regularity1}, we can also show that 
$$\lim_{j \ra \infty} \int_{B_1^c(x)} g(y) \abs{f_j(u_j(y))-f_0(\tilde{u}(y))} \, \dy = 0.$$
Therefore, from \eqref{Riesz3} we conclude $u_j \ra \tilde{u}$ in $L^{\infty}(B_1^c)$. Moreover, $u_j(x) \ra \tilde{u}(x)$ (up to a subsequence) a.e. in $\overline{B_1}$. Hence by the Egorov's theorem, there exists $A \subset \overline{B_1}$ with $|\overline{B_1} \setminus A|=0$ such that $\norm{u_j-\tilde{u}}_{L^{\infty}(A)} \ra 0$. Thus, $u_j \ra \tilde{u}$ uniformly on $B^c_1 \cup A$ where $|(B^c_1 \cup A)^c|=0$. 

\noi  \textbf{\underline{Step 2:}} In this step, we show that $\tilde{u} >0$ on $\RN$. Then the non negativity of $u_j$ follows from the uniform convergence of $u_j$ into $\tilde{u}$. Since $\tilde{u}$ is a solution of \eqref{P2}, we have $\tilde{u} \ge 0$ a.e. in $\RN$ and also using Proposition \ref{fundamental},
\begin{align*}
    -\De \tilde{u}(x) = \mathcal{R}_2  \intRn \frac{g(y)f_0(\tilde{u}(y))}{|x-y|^{N-2}} \, \dy \; \text{ a.e. in } \RN. 
\end{align*}
Therefore, $-\De \tilde{u}  \geq 0$ a.e. in $\RN$. Moreover, from Proposition \ref{local}-$(ii)$ we have $\tilde{u} \in W_{loc}^{1,2}(\RN)$. Now, using the strong maximum principle \cite[Proposition 3.2]{KLP07}, we infer that 
\begin{align}\label{STM1}
    \text{ either }  \tilde{u} \equiv 0   \text{ or }  \tilde{u}>0  \text{ on } \RN.        \end{align}
Since $|(B^c_1 \cup A)^c|=0$, from Proposition \ref{lowerbound} there exists a positive constant $\be_2$ (independent of $j$) such that $\norm{u_j}_{L^{\infty}(B^c_1 \cup A)} \ge \be_2$ for all $j \in \N$. Further, since $u_j \ra \tilde{u}$ in $L^{\infty}(B^c_1 \cup A)$ (see Step 1), $\norm{\tilde{u}}_{L^{\infty}(B^c_1 \cup A)} \ge \be_3$ for some positive constant $\be_3$. Therefore, using \eqref{STM1} we must have $\tilde{u}>0$ on $\RN$. Thus, again from the uniform convergence of $(u_j)$ to $\tilde{u}$, there exists $j_1 \in \N$ such that for $j \ge j_1$, $u_j \ge 0$ on $B^c_1 \cup A$. Moreover, since $u_j \in C(\RN)$ we get  $u_j \ge 0$ on $\RN$.

\noi  \textbf{\underline{Step 3:}}  Since $f_j$ is locally Lipschitz (from \ref{f4}) and $0 \le u_j, u \le C$, we have $|f_j(u_j(y)) - f_j(\tilde{u}(y))| \le M |u_j(y)-\tilde{u}(y)|$ for some $M>0$.  Using \eqref{Riesz3} for every $x\in B^c_1$ we write
\begin{align*}
     \abs{u_j(x) - \tilde{u}(x)} & \le M \mathcal{R}_4  \intRn \frac{g(y)|u_j(y)-\tilde{u}(y)|}{|x-y|^{N-4}} \, \dy + a_j \mathcal{R}_4 \intRn \frac{g(y)}{|x-y|^{N-4}} \, \dy.
\end{align*}
Now 
\begin{align*}
  \intRn \frac{g(y)|u_j(y)-\tilde{u}(y)|}{|x-y|^{N-4}} \, \dy \le \norm{u_j-\tilde{u}}_{\infty} \intRn \frac{g(y)}{|x-y|^{N-4}} \, \dy.
\end{align*}
Therefore, using \ref{g2},
\begin{equation}\label{Riesz4}
    \begin{split}
        \abs{u_j(x) - \tilde{u}(x)} & \le \mathcal{R}_4 \left( M \norm{u_j-\tilde{u}}_{\infty} + a_j \right) \intRn \frac{g(y)}{|x-y|^{N-4}} \, \dy \\
        & \le \mathcal{R}_4 \left( M \norm{u_j-\tilde{u}}_{\infty} + a_j \right) \frac{C_g}{|x|^{N-4}}.
    \end{split}
\end{equation}
Hence 
\begin{align}\label{limit}
    \underset{{x \in B^c_1}}{\sup}\left\{ |x|^{N-4} \abs{u_j(x) - \tilde{u}(x)} \right\} \ra 0, \text{ as } j \ra \infty.
\end{align}
Now we show that $\underset{|x| \ra \infty}{\lim} |x|^{N-4} \tilde{u}(x)>0$. Using \eqref{Riesz5} we get
\begin{equation}\label{Riesz9}
    \begin{split}
        \lim_{|x| \ra \infty} |x|^{N-4} \tilde{u}(x) & = \mathcal{R}_4 \lim_{|x| \ra \infty} \intRn \frac{g(y)f_0(\tilde{u}(y))|x|^{N-4}}{|x-y|^{N-4}} \, \dy  \\
    & \ge \mathcal{R}_4 \lim_{|x| \ra \infty}  \int_{B_R} \frac{g(y)f_0(\tilde{u}(y))|x|^{N-4}}{|x-y|^{N-4}} \, \dy,
    \end{split}
\end{equation}
for any $R>0$. Choose $R>0$ arbitrarily. Then there exists $x \in \RN$ such that $|x|>2R+1$. Hence 
\begin{align*}
    |x-y|^{N-4} \ge \abs{|x|-|y|}^{N-4} \ge \abs{|x|-R}^{N-4} \ge 2^{4-N} \left(1+ \abs{x} \right)^{N-4}, \; \text{ for } y \in B_R.
\end{align*}
Using the above estimate, for $y \in B_R$ we get
\begin{align*}
    \frac{g(y)f_0(\tilde{u}(y))|x|^{N-4}}{|x-y|^{N-4}} \le 2^{N-4} \frac{g(y)f_0(\tilde{u}(y))|x|^{N-4}}{\left(1+ \abs{x} \right)^{N-4}} \le 2^{N-4} g(y)f_0(\tilde{u}(y)). 
\end{align*}
Further, $\frac{g(y)f_0(\tilde{u})|x|^{N-4}}{|x-y|^{N-4}} \ra g(y)f_0(\tilde{u})$ a.e. in $B_R$, as $|x| \ra \infty$. Therefore, by the dominated convergence theorem, 
\begin{align*}
    \lim_{|x| \ra \infty}  \int_{B_R} \frac{g(y)f_0(\tilde{u}(y))|x|^{N-4}}{|x-y|^{N-4}} \, \dy = \int_{B_R} g(y)f_0(\tilde{u}(y)) \, \dy.
\end{align*}
Hence from \eqref{Riesz9} we conclude that 
\begin{align*}
    \lim_{|x| \ra \infty} |x|^{N-4} \tilde{u}(x) \ge \mathcal{R}_4 \int_{B_R} g(y)f_0(\tilde{u}(y)) \, \dy.
\end{align*}
Since $R>0$ is arbitrary, using Fatous lemma we have
\begin{align*}
    \lim_{|x| \ra \infty} |x|^{N-4} \tilde{u}(x) \ge \mathcal{R}_4 \int_{\R^N} g(y)f_0(\tilde{u}(y)) \, \dy.
\end{align*}
Further, since $\tilde{u}>0$, from \eqref{Riesz5}
it follows that $gf_0(\tilde{u}) \gneqq 0$ on $\R^N$. Hence,  $\lim_{|x| \ra \infty} |x|^{N-4} \tilde{u}(x) >0.$
Therefore, from \eqref{limit} there exist $j_2 \in \N$ and $R_1>>1$ such that for $j \ge j_2$, $u_j > 0$ on $B_{R_1}^c$. Since $\tilde{u}, u_j \in C(\RN)$, using Step 1 we have $u_j \ra \tilde{u}$ in  $C((\overline{B_{R_1}} \setminus B_1) \cup A)$  with respect to the sup norm, and by Step 2, $\tilde{u} > \eta$ on $\overline{B_{R_1}}$ for some $\eta >0$. Further, note that $(\overline{B_{R_1}} \setminus B_1) \cup A \subset \overline{B_{R_1}}$.  Hence there exists $j_3 \in \N$ such that for $j \ge j_3$, $u_j>0$ on $(\overline{B_{R_1}}\setminus B_1) \cup A$. Set $\tilde{j}= \max\{ j_2,j_3\}$. Therefore, for $j \ge \tilde{j},$ $u_j>0$ on $B_1^c \cup A$ where $|(B^c_1 \cup A)^c|=0$. Thus, for $j \ge \tilde{j},$ $u_j>0$ a.e. on $\RN$. This completes our proof.
\end{proof}

\noi \textbf{Proof of Theorem \ref{result}:} Proof of $(i)$ follows combining Theorem \ref{Existence} and Proposition \ref{Regularity2}. Proof of $(ii)$ follows using Proposition \ref{bound solution} and Proposition \ref{Regularity1}. Proof of $(iii)$ follows from Theorem \ref{positive}.

\begin{example}\label{f and g}
Let $N \ge 5$ and $\de \ge 1$. Consider the following functions 
\begin{align*}
    f(t)=2t\ln (1+|t|); \quad g(y)= \chi_{A} |y|^{-d} \text{ where } \overline{A} \subset B_1 \setminus B_{\frac{1}{2}}, d>0.
\end{align*}
$(i)$ One can verify that $f$ satisfies \ref{f1}-\ref{f4}. \\
\noi $(ii)$ It is easy to see $g \in L^1(\RN) \cap L^{\infty}(\RN)$. Let $x \in B_1^c$ and $y \in A$. Then $|x-y| \ge d(A,B_1^c)>0$. Hence
\begin{align*}
    |x|^{(N-4) \de} \intRn \frac{g(y)^{\de}}{|x-y|^{(N-4)\de}} \, \dy & \le \int_{A} \frac{|y|^{-d \de}}{|x-y|^{(N-4)\de}} \, \dy \\
    & \le {d(A,B_1^c)}^{(4-N)\de} \int_{B_1 \setminus B_{\frac{1}{2}}} |y|^{-d \de} \, \dy \\
    & \le C.
\end{align*}
Thus $g$ satisfies \ref{g2}.
\end{example}

\begin{remark}
$(i)$ Let $g \in L^1(\RN) \cap L^{\infty}(\RN)$. Suppose $g$ satisfies the following assumption:
\begin{align}\label{g2a}
    |x|^{N-4} \intRn \frac{g(y)}{|x-y|^{N-4}} \, \dy \le C_g, \quad \text{for } x \in \RN \setminus \{ 0 \}, \text{ where } C_g>0.
\end{align}
Observe that, unlike the condition \ref{g2}, here we take $\de =1$ and domain as $\RN \setminus \{ 0 \}$. In this case, if $f$ satisfies \ref{f1} for $\ga < 2^{**}-1$, and \ref{f2}-\ref{f4}, then using the similar set of arguments as given in this article we can show that $u_a$ is uniformly bounded in $L^{\infty}(\RN)$ for all $a \in (0,a_2)$, and $u_a$ is positive on $\RN$ for all $a \in (0, a_3)$. 

$(ii)$ We stress that if $g$ satisfies \eqref{g2a} and $f$ satisfies \ref{f1} for $\ga \in [2^{**}-1, 2^{**})$, then the arguments given in Proposition \ref{Regularity1} fail to provide the uniform boundedness of $u_a$ (required for our proof of the positivity of $u_a$). Also, in order to consider $\ga \in (2,2^{**}-1 )$, we first require $2<2^{**}-1$ which holds only when $N < 12$ (as $N \ge 5$). Thus, under the assumption \eqref{g2a} we get the positivity of solutions to \eqref{SP} for restricted dimensions. 
\end{remark}

\section*{Acknowledgement}
The third author was supported by the DST-INSPIRE Grant DST/INSPIRE/04/2018/002208. Part of the research of the first and second authors was supported by the DST-INSPIRE Research Grant DST/INSPIRE/04/2018/002208 of A.S. The second author also acknowledges the support of the Israel Science Foundation (grant 637/19) founded by the Israel Academy of Sciences and Humanities.

\bibliographystyle{plainurl}

\begin{enumerate}
    \item[E-mail:] nirjan22@tifrbng.res.in
    \item[E-mail:] ujjal.rupam.das@gmail.com
    \item[E-mail:] abhisheks@iitj.ac.in
\end{enumerate}
\end{document}